\documentclass[11pt]{article}
\usepackage{verbatim} 
\usepackage{graphicx}
\usepackage[utf8]{inputenc}
\usepackage{amsmath,amsfonts,amssymb,amscd,amsthm,txfonts,hyperref,float}
\usepackage{xcolor}
\newtheorem{theorem}{Theorem}[section]
\newtheorem{proposition}[theorem]{Proposition}
\newtheorem{lemma}[theorem]{Lemma}
\newtheorem{corollary}[theorem]{Corollary}
\newtheorem{remark}{Remark}[section]

\theoremstyle{definition}
\newtheorem{definition}[theorem]{Definition}
\newtheorem{notation}[theorem]{Notation}

\newcommand{\R}{\mathbb{R}}
\newcommand{\C}{\mathbb{C}}
\newcommand{\N}{\mathbb{N}}
\newcommand{\Z}{\mathbb{Z}}
\newcommand{\T}{\mathbb{T}}

\newcommand{\E}{\mathbb E}

\newcommand{\an}[1]{\langle #1 \rangle}

\begin{document}

\title{General remarks on the propagation of chaos in wave turbulence and application to the incompressible Euler dynamics}
\author{Anne-Sophie de Suzzoni\footnote{CMLS, \'Ecole Polytechnique, CNRS, Universit\'e Paris-Saclay, 91128 PALAISEAU Cedex, France, \texttt{anne-sophie.de-suzzoni@polytechnique.edu}}}

\maketitle

\begin{abstract}
In this paper, we prove propagation of chaos in the context of wave turbulence for a generic quasisolution. We then apply the result to full solutions to the incompressible Euler equation.
\end{abstract}

\section{Introduction}

We address the question of propagation of chaos in the context of \emph{wave} turbulence.

The issue at stake is the following : we consider the solution to a Hamiltonian equation with a random initial datum whose Fourier coefficients are initially independent and we want to know if this independence remains satisfied at later times. These Fourier modes must satisfy what is called in the Physics literature \emph{Random Phase Approximation}, which is something satisfied by Gaussian variables. Here, we address also the following question : assuming that the initial Fourier modes are Gaussian, do the Fourier modes at later times conserve some sort of Gaussianity.

In the context of weak turbulence and for Schrödinger equations, these questions have been successfully adressed by Deng and Hani in \cite{denghani21}. The Gaussianity in these papers consists in proving that at later times the moments of the Fourier modes still behave like Gaussian moments. 

Of course, the independance and Gaussianity are asymptotic in some sense. In the work by Deng and Hani, the cubic Schrödinger equation is considered on a torus of size $L\gg 1$ and with an initial datum of size $\varepsilon(L) \ll 1$ but at very big times in terms of $\varepsilon$, passed the deterministic nonlinear time, at the so-called kinetic time, where nonlinear effects start appearing in the dynamics of the statistics. They prove that the correlations between different Fourier modes tend to $0$ as $L \rightarrow \infty$ and that if the initial datum is a Gaussian field, then the Fourier conserve Gaussian moments. They deduce this result from their successful derivation of the so-called \emph{kinetic} equation, see \cite{DH21}. 

Here, we do not adress the issue of the derivation of the kinetic equation. However, we mention the pioneer work by Peierls, \cite{peierls}, the following works by Brout and Prigogine or Prigogine alone, \cite{Brout-Prigo, Prigogine}, and the works on fluid mechanics by Hasselman \cite{Hass1,Hass2}, Zakharov and Filonenko or Zakharov alone, \cite{Zakharov1967,ZakFil66,KZspectra}. For a review, we mention the book by Nazarenko, \cite{Naz}. Mathematical works on the derivation of kinetic equations for the Schrödinger equations include \cite{CoGe20,denghani19,CoG19,ACG21,BGHS,DyKuk1,DyKuk2,DyKuk3}. For Korteweg de Vries type equations, we mention \cite{ST21}. Finally, we mention a result on discrete Schrödinger equations \cite{LukSpo}. 

In this paper, what we call asymptotic Gaussianity is the fact that the formula of cumulants remain asymptotically valid at later times. 

In the first part of this paper, we adress these issues on a generic Hamiltonian equation. We work, similarly to \cite{DyKuk1,DyKuk2,DyKuk3}, and to \cite{Naz}, in the context of \emph{quasi-solutions} and of wave turbulence. We do not assume that the initial datum is somewhat small but we let the size of the torus go to $\infty$ and this is our asymptotic regime. The proof is mainly combinatorial.

In the second part of this paper, we pass from quasisolutions to full solutions to the incompressible Euler equation. For this part, we need a functional framework that fits both the initial datum and the Euler equation. We adapt the analytic functional framework of \cite{B-GCT11} keeping in mind that for our problem the initial datum is not localised. We also need to render explicit the abstract Cauchy-Kowaleskaia theorem, and for this, we use \cite{Caf90}. Finally, we need to estimate probabilities on the initial datum, we use a strategy very close to proving Fernique's theorem, see \cite{fernique}.

\subsection{Framework and results}

We consider a generic equation : 
\begin{equation}\label{genEq}
\partial_t u_L = K u_L + J_L (u_L,\hdots ,u_L) 
\end{equation}
on the torus $L\T^d$ of size $L$ and in dimension $d$. Here, $K$ is a skew-symmetric operator, and $J_L$ a $N$-linear map, the map $u$ has values in $\C^D$.

We assume that $K$ and $J_L$ take the following form in Fourier mode : for any test functions, $u_L,u_{L,1},\hdots,u_{L,N}$, we set
\begin{multline}\label{defOmegaPsi}
\widehat{Ku_L} (\xi) = i\omega(\xi) \hat u_L(\xi)\\
\widehat{J_L(u_{L,1},\hdots,u_{L,N})}(\xi) = \frac1{(2\pi L)^{d(N-1)/2}}\sum_{\xi_1+\hdots+\xi_N = \xi} \Psi(\xi_1,\hdots,\xi_N) (\hat u_{L,1}(\xi_1), \hdots, \hat u_N(\xi_{L,N}))
\end{multline}
where $\Psi(\xi_1,\hdots, \xi_N)$ is a $N$-linear map from $(\C^D)^N$ to $\C^D$, where for all $\xi \in \frac1{L} \Z^d$,
\[
\hat u_L (\xi) := \frac1{(2\pi L)^{d/2}} \int_{L\T^d} u(x) e^{-i\xi x}dx.
\]
We also assume that for all $u_{L,1},\hdots u_{L,N}$, 
\[
\widehat{J_L(u_{L,1},\hdots,u_{L,N})}(0) = 0
\]
such that the quantity
\[
\int_{L\T^d} dx u(x)
\]
is conserved under the action of the flow of \eqref{genEq} and thus can be chosen null.

Finally, we assume that $\Psi$ has at most linear growth : there exists $r\in [0,1]$ such that for all $(\xi_1,\hdots ,\xi_N) \in (\R^{d})^N$, in operator norm
\[
|\Psi (\xi_1,\hdots ,\xi_N)| \leq \max_{j=1}^N \an{\xi_j}^r.
\]

We set the following initial datum for \eqref{genEq} :
\begin{equation}\label{InitialDatum}
u(t=0)(x) = a_L (x) := \sum_{k\in \Z^d_*} \frac{e^{ikx/L}}{(2\pi L)^{d/2}} g_k a_{L,k}
\end{equation}
where $\Z^d_* = \Z^d\smallsetminus \{(0,\hdots,0)\}$. We write $u_L$ the solution to \eqref{genEq} with initial datum $a_L$.

In \eqref{InitialDatum}, $(g_k)_{k\in \Z^d_*}$ is a sequence of centred and normalized complex Gaussian variables such that for all $k\in \Z^d_*$, 
\[
g_k = \bar g_{-k}, 
\]
and such that if $k\neq l, -l$, then $g_k$ and $g_l$ are independent.

Finally, $(a_{L,k})_k$ is a sequence with values in $(\R^D)$ with finite support such that $a_{L,-k} = a_{L,k}$ for all $k\in \Z^d_*$.

We define by induction for $n\in \N$, $t\in \R$,
\begin{equation}\label{Picard expansion}
u_{L,0}(t) = e^{tK} a_L, \quad u_{L,n+1} (t) = \sum_{n_1+\hdots n_N = n} \int_{0}^t e^{(t-\tau)K}[J(u_{L,n_1}(\tau), \hdots, u_{L,n_N}(\tau))] d\tau.
\end{equation}

For $M \in \N$, 
\[
\sum_{n=0}^M u_{L,n}
\]
is called a quasi-solution. 

For a given $\xi \in \frac1{L}\Z_*^d$ and a given $t$, $\hat u_{L,n}(t)(\xi)$ is a vector in $\C^D$, we write $\hat u_{L,n}^{(i)}(t)(\xi)$ its $i$-th component.

\begin{remark}\label{rem:spaceInv} We note that the law of the initial datum is invariant under the action of space translations. For any space translation $\tau$, we also have 
\[
\tau K = K \tau,\quad \tau J(\cdot,\hdots,\cdot) = J(\tau \cdot,\hdots, \tau \cdot).
\]
Therefore, by induction on $n$ the law of $(u_{L,n})_n$ is invariant under space translations and therefore, for all $n,m,i,j,t$, 
\[
\E( u_{L,n}^{(i)}(t)(\xi) u_{L,m}^{(j)}(t)(\eta))
\]
is equal to $0$ unless $\eta = -\xi$.
\end{remark}

In this framework, we prove Theorem \ref{th:genEq}.

\begin{theorem}\label{th:genEq}
There exists $C = C(\Psi,N)$ such that for all $R\in \N^*$, $(n_1,\hdots,n_R) \in \N^R$, $(i_1,\hdots, i_R) \in [|1,d|]]^R$, $(\xi_1,\hdots,\xi_R) \in (\frac1{L}\Z_*^d)^R$, all $t\in \R$, we have 
\begin{equation}\label{EstTh}
\Big| \E\Big(\prod_{l=1}^R \hat u_{L,n_l}^{(i_l)}(t)(\xi_l)\Big) - \sum_{\mathcal O \in \mathcal P_R} \prod_{\{l,l'\}\in \mathcal O} \E(\hat u_{L,n_l}^{(i_l)}(t)(\xi_l) \hat u_{L,n_{l'}}^{(i_{l'})}(t)(\xi_{l'}))\Big| \leq \frac{S!}{(S/2)!} \|(a_{L,k})_k\|_{\ell^2\cap \ell^\infty} \frac{(CA_L^r)^{\sum n_l}}{(2\pi L)^{d/2}}
\end{equation}
if $S =  \sum_{l}n_l(N-1) + R$ is even, otherwise 
\[
\E\Big(\prod_{l=1}^R \hat u_{L,n_l}^{(i_l)}(t)(\xi_l)\Big) = \sum_{\mathcal O \in \mathcal P_R} \prod_{\{l,l'\}\in \mathcal O} \E(\hat u_{L,n_l}^{(i_l)}(t)(\xi_l) \hat u_{L,n_{l'}}^{(i_{l'})}(t)(\xi_{l'})) .
\]
Above we used the notations
\[
A_L = \sup \{\an{\frac{k}{L}} \;|\; k\in \Z^d_*, a_{L,k}\neq 0\},
\]  
the set $\mathcal P_R$ is the set of partitions of $[|1,R|]$ that contains only pairs (hence it is empty if $R$ is odd).
\end{theorem}

\begin{remark} Taking $a_{L,k} = a(\frac{k}{L})$ where $a$ is a bounded, compactly supported function, we have 
\[
\|(a_{L,k})_k\|_{\ell^2\cap \ell^\infty} \lesssim \|a\|_{L^2\cap L^\infty(\R^d)}
\]
and 
\[
A_L \leq A_\infty = \sup \{\an\xi \; |\; a(\xi)\neq 0\}.
\]
Hence in this context the difference in \eqref{EstTh} is a $\mathcal O(L^{-d/2})$.
\end{remark}

\begin{remark} This theorem contains the asymptotic formula of cumulants for the quasisolutions, but considering Remark \ref{rem:spaceInv}, it also implies asymptotic independence.
\end{remark}

In the context of the Euler incompressible equation : 
\begin{equation}\label{Euler}
\left \lbrace{\begin{array}{c}
\partial_t u_L + u_L\cdot \nabla u_L = P_L \\
\nabla \cdot u_L =0 \\
u_L(t=0) = a_L 
\end{array}}\right.
\end{equation}
where $P_L$ is the pressure and $\nabla \cdot$ is the divergence, we assume that the sequences $(a_{L,k})_k$ take the form :
\[
a_{L,k} = \varepsilon(L) a(\frac{k}{L})
\]
where $\varepsilon(L) = O(\frac1{\sqrt{\ln L}})$, such that $\varepsilon^{-1}$ has at most polynomial growth in $L$ and where $a$ is a bounded, compactly supported function. In order to have initially $\nabla \cdot u_L(t=0) = 0$, we impose that for all $\xi \in \R^d$, $\xi\cdot a(\xi) = 0$. We prove (local) well-posedness of \eqref{Euler} in the analytical framework presented in Subsection \ref{subsec:wellposedness}. In this analytical framework, the size of the initial datum can be up to $\varepsilon(L) \sqrt{\ln L}$, we refer to Appendix \ref{app:sizeofaL}. But if one looks at the initial datum locally, it is as small as $\varepsilon (L)$. Indeed, we have that for a given $x\in L\T^d$, the random variable
\[
\sum_{k\in \Z^d_*} \frac{e^{ikx/L}}{(2\pi L)^{d/2}} g_k a(k/L)
\]
converges in law towards the Wiener integral
\[
\frac1{(2\pi)^{d/2}}\int e^{i\xi x} a(\xi) dW(\xi)
\]
where $W$ is a multidimensional Brownian motion. However, the regime we impose on $\varepsilon(L)$ is quite different that the ones in \cite{denghani21,DH21}, which are imposed by the dispersion of the Schrödinger equation. What is more, we do not claim that we reach derivation of the kinetic equation, or that we reach kinetic times. The result is valid for any time if $\varepsilon (L) = o((\ln L)^{-1/2})$, or for small times if $\varepsilon(L) = O((\ln L)^{-1/2})$ but we do not rescale the time. In this context, we prove Theorem \ref{th:Euler}.

\begin{theorem}\label{th:Euler} There exist Banach spaces $(\mathcal{X}, \|\cdot\|_\chi)$ and $\mathcal Y_\theta$ such that 
\[
\mathcal X  \subseteq \mathcal C(\R^d, \C^d), \quad \mathcal Y_\theta \subseteq \mathcal C([-\theta,\theta]\times \R^d, \C^d)
\]
such that for all $\theta \in \R_+$, there exists $A(\theta) >0$ such that the Cauchy problem
\[
\left \lbrace{\begin{array}{c}
\partial_t u + u\cdot \nabla u = P \\
\nabla \cdot u =0 \\
u(t=0) = u_0 
\end{array}}\right.
\]
is well-posed in $\mathcal Y_\theta$ for all $u_0$ in the ball of $\mathcal X$ of center $0$ and radius $A(\theta)$. The map $\theta \mapsto A(\theta)$ can be chosen nonincreasing. The flow hence defined conserves periodicity.

What is more, seeing $a_L$ as a periodic function of $\R^d$ we get that there exists $c>0$ such that if $A(\theta) \geq \frac{\sqrt{\ln L}\varepsilon(L)}{c}$, we have that $a_L$ belongs to $\chi$ and 
\[
\mathbb P(\|a_L\|_{\mathcal X}> A(\theta))\leq e^{-cA(\theta)^2\varepsilon^{-2}(L)}.
\]
Writing 
\[
\mathcal E_{L,\theta} = \{\|a_L\| \leq A(\theta)\},
\]
we get that for all $R\in \N^*$, there exists $C = C(R,a,\theta, \mathcal X, \varepsilon) $ (the constant depends on the functional framework and the function $\varepsilon$ but not on $L$) and $c_1 = c_1(a,\theta,R)$, $c_2 = c_2(a,\theta,\mathcal X)$ such that for all $(i_1,\hdots, i_R) \in [|1,d|]]^R$, $(\xi_1,\hdots,\xi_R) \in (\frac1{L}\Z_*^d)^R$, all $t\in [-\theta,\theta]$, and for all $L$, assuming
\[
\varepsilon(L)\sqrt{\ln L} \leq c_1(a,\theta,R),\quad A(\theta) \geq \frac{\sqrt{\ln L}\varepsilon(L)}{c_2},
\]
we have
\begin{equation}\label{EstEuler}
\Big| \E\Big({\bf 1}_{\mathcal E_{L,\theta}}\prod_{l=1}^R \hat u_{L}^{(i_l)}(t)(\xi_l)\Big) - \sum_{\mathcal O \in \mathcal P_R} \prod_{\{l,l'\}\in \mathcal O} \E({\bf 1}_{\mathcal E_{L,\theta}}\hat u_{L,n_l}^{(i_l)}(t)(\xi_l) \hat u_{L,n_{l'}}^{(i_{l'})}(t)(\xi_{l'}))\Big| \leq C \varepsilon(L)^R L^{-d/2} .
\end{equation}
What is more, if $\xi_l \neq -\xi_{l'}$, we have 
\[
\E({\bf 1}_{\mathcal E_{L,\theta}}\hat u_{L,n_l}^{(i_l)}(t)(\xi_l) \hat u_{L,n_{l'}}^{(i_{l'})}(t)(\xi_{l'})) \lesssim_{a,\theta,\mathcal X,\varepsilon} \varepsilon(L)^R L^{-d/2}.
\]
\end{theorem}

\begin{remark} If $\varepsilon(L) = o((\ln L)^{-1/2})$, then the result is global, because the inequalities 
\[
A(\theta) \geq \frac{\sqrt{\ln L}\varepsilon(L)}{c_2}, \quad \varepsilon(L) \sqrt{\ln L} \leq c_1
\]
are satisfied for $L$ big enough. Otherwise, we need,
\[
A(\theta) > \frac{\limsup (\varepsilon(L) \sqrt{\ln L})}{c_2},
\]
which requires that $\theta$ has to be small enough. In other words, if $\limsup \varepsilon(L) \sqrt{\ln L}= c$, we need both $\theta$ to be smaller than a constant depending on the functional framework, the function $a$ and $c$ (non-increasing with $c$). But we also need that $c$ is smaller than a constant depending on $R,a$ and the functional framework.
\end{remark}

\subsection{Notations}

By $\an{\cdot}$, we denote the Japanese bracket, that is for $x\in \R^d$, 
\[
\an{x} = \sqrt{1 + \sum_{i=1}^d x_i^2}.
\]

By $[|a,b|]$ with $a\leq b \in \R$, we denote $[a,b] \cap \N$.

By the lexicographical order on $\N^2$, we mean the order defined for $(l_1,j_1)$ and $(l_2,j_2) \in \N^2$ as 
\[
(l_1,j_1) < (l_2,j_2) \quad \Leftrightarrow \quad l_1 < l_2 \textrm{ or } (l_1=l_2 \textrm{ and } j_1<j_2).
\]

For the norms on the sequence $(a_{L,k})_k$, we denote
\[
\|(a_{L,k})_k\|_{\ell^\infty} = \sup_{k\in \Z^d_*} |a_{L,k}|, \quad \|(a_{L,k})_k\|_{\ell^2} = \frac1{(2\pi L)^{d/2}} \sqrt{\sum_{k\in \Z^d_*} |a_{L,k}|^2}, \quad A_L = \sup \{\an{\frac{k}{L}} \; | \; a_{L,k}\neq 0\}, 
\]
such that if $a_{L,k} = a(k/L)$ with $a\in L^\infty$ with compact support, setting
\[
 A_\infty = \sup\{\an{\xi} \; |\; a(\xi) \neq 0 \},
\]
we have, for all $L$,
\[
A_L \leq A_\infty, \quad \|(a_{L,k})_k\|_{\ell^\infty} \leq \|a\|_{L^\infty}, \quad \|(a_{L,k})_k\|_{\ell^2} \leq A_\infty^{d/2}\pi^{-d/2}\|a\|_{L^\infty}
\]
when $L \rightarrow \infty$. We also denote
\[
\|(a_{L,k})_k\|_{\ell^\infty\cap \ell^2} = \|(a_{L,k})_k\|_{\ell^\infty}+\|(a_{L,k})_k\|_{ \ell^2}.
\]

The spaces $L^p(\R^d)$ are the standard Lebesgue spaces.

Finally, in all the paper but Subsection \ref{subsec:wellposedness}, we consider Fourier transforms for $L$-periodic functions, or for functions of the torus $L\T^d$. We use the previously mentioned convention
\[
\hat u_L (\xi) = \frac1{(2\pi L)^{d/2}} \int_{L\T^d} u_L(x) e^{-i\xi x} dx
\]
for $u_L $ defined on $L\T^d$, $\xi \in \frac1{L}\Z_*^d$. With this convention, we have
\[
\hat a_L(\frac{k}{L}) = a_{L,k} g_k.
\]
When $u_L$ also depends on time, we set for all $t\in \T$, $\hat u_L (t)(\xi) = \widehat{u_L(t)} (\xi)$.

In Subsection \ref{subsec:wellposedness}, we consider functions of the full $\R^d$, without conditions of periodicity, we use the convention that the Fourier transform of a Schwartz class function $f$ at $\xi\in \R^d$ is defined as
\[
\frac1{(2\pi)^{d/2}} \int_{\R^d} f(x)e^{-ix\xi} dx.
\]

\subsection{Acknowledgements}

The author thanks Nikolay Tzvetkov for suggesting to study the Euler equation and providing helpful literature.

The author is supported by ANR grant ESSED ANR-18-CE40-0028.

\section{Asymptotic independence of the quasi solutions}

\subsection{N-trees}

We introduce the notion of $N$-trees. 

\begin{definition}\label{def:Ntrees} Let $\mathcal A_0 = \{()\}$ and define by induction for all $n\in \N$,
\[
\mathcal A_{n+1} = \{(A_1,\hdots,A_{N}) \; |\; \forall j, A_j \in \mathcal A_{n_j},\; \sum_{n_j}= n\}
\]
We call the elements in $\mathcal A_n$ the $N$-trees with $n$ nodes. We call $()$ the trivial tree.
\end{definition}

\begin{remark} A $N$-tree is a sequence of parenthesis and commas. Another way of defining $N$-trees is to use Polish notation and write
\[
\mathcal A_0 = \{0\}, \quad \mathcal A_{n+1} = \{1 A_1\hdots A_N\; |\;  \forall j, A_j \in \mathcal A_{n_j},\; \sum_{n_j}= n\}
\]
and see the $N$-trees as sequences of $0$ and $1$. In this case, the decomposition $1A_1\hdots A_N$ is unique (see Appendix \ref{app:Polish}).
\end{remark}

\begin{proposition}\label{prop:PicardExp} Define by induction on the $N$-trees, for all $t\in \R$, $A\in \cup_n \mathcal A_n$,
\[
F_{L,A} (t) = \left \lbrace{\begin{array}{cc}
u_{L,0}(t) & \textrm{ if } A = () \\
\int_{0}^t e^{(t-\tau)K}[J_L(F_{L,A_1}(\tau),\hdots,F_{L,A_N}(\tau))]d\tau & \textrm{ if } A=(A_1,\hdots,A_N)
\end{array}} \right. .
\]
We have for $n\in \N$,
\[
u_{L,n}(t) = \sum_{A\in \mathcal A_n} F_{L,A}(t).
\]
\end{proposition}

\begin{proof} The proof follows by induction on $n$. For $n=0$, this is by definition. Otherwise, we have 
\[
u_{L,n+1}(t) = \sum_{n_1+\hdots+n_N =n } \int_{0}^t e^{(t-\tau)K}[J_L(u_{L,n_1}(\tau), \hdots,u_{L,n_N}(\tau))] d\tau.
\]
Using the induction hypothesis and the fact that all the sums are finite, we get
\[
u_{L,n+1}(t) =\sum_{n_1+\hdots+n_N =n } \sum_{A_j \in \mathcal A_{n_j}} \int_{0}^t e^{(t-\tau)K}[J_L(F_{L,A_1}(\tau),\hdots,F_{L,A_N}(\tau))]d\tau.
\]
We recognize
\[
u_{L,n+1}(t) =\sum_{n_1+\hdots+n_N =n } \sum_{A_j \in \mathcal A_{n_j}} F_{L,(A_1,\hdots,A_N)}(t)
\]
and we use the definition of $N$-trees to conclude.
\end{proof}

\begin{definition}\label{def:labelling} Let $n\in \N$ and $\vec k = (k_1,\hdots,k_{(N-1)n+1}) \in ( Z^{d}_*)^{(N-1)n+1}$. Let $A \in \mathcal A_n$. Define $F_{L,A,\vec k}$ by induction on $n$ in the following way. If $n=0$ then $A = ()$ and $\vec k = (k_1)$, we set
\[
F_{L,A,\vec k}(t) = e^{it\omega (\frac{k_1}{L})} g_{k_1} a_{L,k_1}.
\]
If $n = m+1$ with $m\in \N$, there exists $n_1, \hdots n_N$ such that $\sum n_j = m$ and $A_j \in \mathcal A_{n_j}$ such that $A= (A_1,\hdots,A_N)$. We set $\tilde n_j = \sum_{l<j} ((N-1)n_l + 1)$ and 
\[
\vec k_j = (k_{(N-1) \tilde n_j + 1},\hdots, k_{\tilde n_{j+1}}) \in  (\Z^d_*)^{n_j(N-1)+1}.
\]
(Note that $\tilde n_{N+1} = \sum_{j=1}^N ((N-1)n_j +1) = (N-1)m + N = (N-1)n +1$.)  We set also $R(\vec k) = \frac1{L}\sum_{j=1}^{(N-1)n+1} k_j$.

We now set
\[
F_{L,A,\vec k}(t) = \frac1{(2\pi L)^{d(N-1)/2}}\int_{0}^t e^{i(t-\tau)\omega(R(\vec k))} \Psi(R(\vec k_1),\hdots,R(\vec k_N))( F_{L,A_1,\vec k_1}(\tau),\hdots, F_{L,A_N,\vec k_N}(\tau)) d\tau.
\]
\end{definition}

\begin{proposition}\label{prop:PicardExpFourier} We have for all $n\in \N$ and $A \in \mathcal A_n$,
\[
\widehat{F_{L,A}(t)}(\xi) = \sum_{R(\vec k) = \xi}  F_{L,A,\vec k}(t) .
\]
\end{proposition}

\begin{remark} The sum is finite because $(a_{L,k})$ has finite support.
\end{remark}

\begin{proof} By induction on $n$. For $n=0$, we have 
\[
F_{L,A,\vec k} (t) = e^{it\omega (R(\vec k))} g_{k_1}a_{L,k_1} = \hat u_0 (t)(R(\vec k)).
\]
For $n=m+1$ with $m\in \N$ with the above construction. We have 
\[
F_{L,A}(t) = \int_{0}^t e^{(t-\tau)K}[J_L(F_{L,A_1}(\tau),\hdots,F_{L,A_N}(\tau))]d\tau.
\]
In Fourier mode, this transforms as
\[
\widehat{F_{L,A}(t)}(\xi) = \frac1{(2\pi L)^{d(N-1)/2}}\int_{0}^t e^{i(t-\tau)\omega(\xi)} \sum_{\xi_1+\hdots \xi_N = \xi} \Psi(\xi_1,\hdots ,\xi_N) ( \widehat{F_{L,A_1}(\tau)}(\xi_1),\hdots, \widehat{F_{L,A_N}(\tau)}(\xi_N)) d\tau.
\]
We use the induction hypothesis to get that
\[
\widehat{F_{L,A_j}(\tau)}(\xi_j) = \sum_{R(\vec k_j) = \xi_j} F_{L,A,\vec k_j}(\tau).
\]
We see now that
\[
\{ \vec k \; | \; R(\vec k_j) = \xi_j \wedge \sum \xi_j= \xi \} = \{ \vec k \;| \; R(\vec k) = \xi\}.
\]
We deduce the result. 
\end{proof}

\begin{proposition}\label{prop:PicardExpFourier2} We have that for all $A\in \mathcal A_n$ and all $\vec k \in (\Z^d_*)^{(N-1)n+1}$,
\[
F_{L,A,\vec k} (t) = \frac1{(2\pi L)^{d(N-1)n/2}} G_{L,A,\vec k}(t) \prod_{j=1}^{(N-1)n+1} g_{k_j}
\]
where $G_{L,(),(k_1)}(t) = e^{i\omega(k_1/L)t}a_{L,k_1}$ and with the notations of Proposition \ref{prop:PicardExpFourier}
\[
G_{L,A,\vec k}(t) = \int_{0}^t e^{i(t-\tau)\omega(R(\vec k))} \Psi(R(\vec k_1),\hdots,R(\vec k_N))( G_{L,A_1,\vec k_j}(\tau), \hdots, G_{L,A_N,\vec k_N}(\tau)) d\tau.
\]
\end{proposition}

\begin{proof} By induction on $n$. 
\end{proof}

Summing up, we have the following formula :

\begin{equation}\label{sumupPicardExp}
\hat u_n(\xi) = \frac1{(2\pi L)^{d(N-1)n/2}}\sum_{A\in \mathcal{A}_n} \sum_{R(\vec k) = \xi} G_{A,\vec k}(t) \prod_{j=1}^{(N-1)n+1} g_{k_j}.
\end{equation}

\subsection{Expectations}

For the rest of this section, we set $R\in \N^*$, $(n_1,\hdots,n_R) \in \N^R$, $i_1,\hdots, i_R \in [|1,D|]^R$ and $(\xi_1,\hdots,\xi_R) \in (\frac1{L} \Z^d_*)^R$. We also set
\[
S = \{(l,j) \; | \; l\in [1,R]\cap \N, j \in [|1,n_l(N-1)+1|] \}
\]
and
\[
\mathfrak S 
\]
the set of involutions of $S$ without fixed points. 

Using Equation \eqref{sumupPicardExp}, we get
\[
\E \Big( \prod_{l=1}^R \hat u^{(i_l)}_{n_l}(t)(\xi_l) \Big) = \frac1{(2\pi L)^{d(N-1)(\sum n_l)/2}} \sum_{A_l \in \mathcal A_{n_l}} \sum_{R(\vec k_l) = \xi_l} \prod_{l=1}^R G^{(i_l)}_{A_l,\vec k_l} (t) \E(\prod_{m\in S} g_{k_m}).
\]

By the formula of cumulants, we have 
\[
\E(\prod_{m\in S} g_{k_m}) = \sum_{\sigma \in \mathfrak S} \prod_{m\in S_\sigma} \E(g_{k_m}g_{k_{\sigma(m)}})
\]
where $S_\sigma = \{ m \in S\;|\; m<\sigma(m) \}$ (using the lexicographical order). We get the following proposition.

\begin{proposition}\label{prop:expectations1} We have that
\[
\E \Big( \prod_{l=1}^R \hat u^{(i_l)}_{n_l}(t)(\xi_l) \Big) = \frac1{(2\pi L)^{d(N-1)(\sum n_l)/2}} \sum_{A_l \in \mathcal A_{n_l}} \sum_{\sigma \in \mathfrak{S}} \sum_{\Sigma_\sigma} \prod_{l=1}^R G^{(i_l)}_{L,A_l,\vec k_l} (t)
\]
where 
\[
\Sigma_\sigma = \{ \vec k \in (\Z^d)^S | \forall l\in [|1,R|], \;R(\vec k_l) = \xi_l,  \; \forall m\in S, k_{m} = -k_{\sigma(m)},
\]
and where
\[
G^{(i_l)}_{L,A_l,\vec k_l} (t) := 0
\]
whenever there exists $j\in [|1,n_l(N-1)+1|]$ such that $k_{l,j} = 0$,
and where we used the notation
\[
\vec k_l = (k_{l,1},\hdots,k_{l,(N-1)n_l+1}).
\]
\end{proposition}

\begin{proof} We have 
\[
\E(g_{k_m}g_{k_{\sigma(m)}}) = \left \lbrace{\begin{array}{cc}
1 & \textrm{ if } k_m = - k_{\sigma(m)} \\
0 & \textrm{ otherwise.}
\end{array}} \right.
\]
\end{proof}

\begin{remark} If the cardinal of $S$, that is, $\sum_l n_l (N-1) + R$ is odd, then the expectation is $0$.
\end{remark}

We also set for $l\in [|1,R|]$, and $j\in [|1,n_l(N-1)+1|]$, 
\[
\sigma(l,j) = (\tilde \sigma(l,j),j')
\]
for some $j' \in [|1,n_{l'}(N-1)+1|]$.

We now compute the dimension of $\Sigma_\sigma$. For this, we introduce the notion of orbits of $\sigma$.

\begin{definition}\label{def:orbits} Let $A\subset [|1,R|]$. We set
\[
\sigma(A) = \{ l \in [|1,R|]\cap \N \; |\; \exists l'\in A,\exists j' \in [|1,n_{l'}(N-1)+1|],  l = \tilde\sigma(l',j')\}.
\]
This defines a map of the parts of $[|1,R|]$ to itself.

We call the orbit of $l$ in $\sigma$ and we write $o_\sigma(l)$ the set
\[
o_\sigma(l) = \bigcup_{n\in \N} \sigma^n(\{l\}).
\]
We write $\mathcal O_\sigma$ the set whose elements are the orbits of $\sigma$.
\end{definition}

\begin{proposition}\label{prop:partition}
The orbits of $\sigma$ form a partition of $[|1,R|]$. 
\end{proposition}

\begin{proof} We prove that the relation $l\in o_\sigma(l')$ is an equivalence relation.  

This relation is reflexive since $l \in \sigma^{0}(\{l\})$ for all $l$.

This relation is symmetric. Indeed, let $l,l' \in [|1,R|]$. We prove that $l \in o(l')$ implies $l' \in o(l)$. Since $l \in o(l')$, there exists $n$ such that $l \in \sigma^n(\{l'\})$. 
Therefore, there exists $j_1, \hdots ,j_n$, $k_0,\hdots,k_{n-1}$ and $l'=l_0,l_1,\hdots, l_{n-1},l_n=l$ such that for all $m=0,\hdots,n-1$,
\[
(l_{m+1},j_{m+1}) = \sigma(l_m,k_m).
\]
Because $\sigma$ is an involution, this also reads as
\[
(l_m,k_m) = \sigma(l_{m+1},j_{m+1})
\]
and thus $l' \in o(l)$.

This relation is transitive. Indeed, if $l \in o_\sigma(l')$ and if $l'\in o_\sigma(l'')$ then there exist $n_1$ and $n_2 \in \N$, such that
\[
l \in \sigma^{n_1}(\{l'\}) \quad l' \in \sigma^{n_2}(\{l''\}).
\] 
Therefore, we have
\[
l \in  \sigma^{n_1+n_2}(\{l''\}) \subseteq o_{\sigma}(l'').
\]
\end{proof}

\begin{proposition}\label{prop:dimSigma} If for all $o \in \mathcal O_\sigma$, we have 
\[
\sum_{l\in o} \xi_l = 0
\]
then
\[
\Sigma_\sigma \sim (Z^d)^{s_\sigma}
\]
with $s_\sigma = \frac12 \# S + \# \mathcal O_\sigma - R$.

Otherwise, $\Sigma_\sigma = \emptyset$.
\end{proposition}

\begin{remark} By $\Sigma_\sigma \sim (Z^d)^{s_\sigma}$, we mean that within the $\# S$ parameters of the elements of $\Sigma_\sigma$, $s_\sigma$ of them are free and $\#S - s_\sigma$ are fixed by the values of the $s_\sigma$ free parameters. More precisely, we mean that up to a reordering of the parameters in $\Sigma_\sigma$,
\[
\Sigma_\sigma = \{ (\xi_1,\hdots,\xi_{s_\sigma}, L_{s_\sigma + 1}(\xi_1,\hdots,\xi_{s_\sigma}),\hdots, L_{\# S}(\xi_1,\hdots,\xi_{s_\sigma}))\; |\; (\xi_1,\hdots,\xi_{s_\sigma})\in \Z^d)^{s_\sigma}
\}
\]
where $L_{s_\sigma +1},\hdots, L_{\# S}$ are linear maps. 
\end{remark}

\begin{proof} For all $l\in [|1,R|] $, set
\[
S_{\sigma,l,+} = \{j \in [1,n_l(N-1)+1]\; |\; l <  \tilde \sigma(l,j)\}
\]
and 
\[
S_{\sigma,l,-} = \{j \in [1,n_l(N-1)+1]\; |\; l > \tilde \sigma(l,j)\}.
\]
Note that $S_{\sigma,l,+}\subseteq S_\sigma$ and that $S_{\sigma, l,-}$ is included in the complementary of $S_\sigma$ in $S$ and that $\sigma(S_{\sigma,l,-})\subseteq S_\sigma$.

By definition, we have 
\[
\Sigma_\sigma = \{ \vec k \in (\Z^d_*)^S | \forall l\in [1,R]\cap \N, \;R(\vec k_l) = \xi_l,  \; \forall m\in S, k_{m} = -k_{\sigma(m)}\}.
\]
By taking only half the $k$s (the ones in $S_\sigma$, the others being entirely determined by the ones in $S_\sigma$), we get
\[
\Sigma_\sigma \sim \Sigma_\sigma' = \{ \vec k \in (\Z^d_*)^{S_\sigma} \;|\; \forall l\in [1,R]\cap \N, \; \sum_{j\in S_{\sigma,l,+}} k_{(l,j)} - \sum_{j \in S_{\sigma,l,-}} k_{\sigma(l,j)} = L \xi_l\}.
\]
Because the orbits of $\sigma$ form a partition of $[1,R]\cap \N$ and because equations 
\[
\sum_{j\in S_{\sigma,l,+}} k_{(l,j)} - \sum_{j \in S_{\sigma,l,-}} k_{\sigma(l,j)} = L \xi_l
\]
involve only $l$s from the same orbit. Indeed, we have that $\tilde \sigma(l,j) \in o_\sigma(l)$. We have the decomposition
\[
\Sigma_\sigma' \sim \prod_{o\in \mathcal O_\sigma} \Sigma_{\sigma,o}
\]
with
\[
\Sigma_{\sigma,o} = \{ \vec k \in (\Z^d_*)^{S_{\sigma,o}} \;|\; \forall l\in o, \; \sum_{j\in S_{\sigma,l,+}} k_{(l,j)} - \sum_{j \in S_{\sigma,l,-}} k_{\sigma(l,j)} = L \xi_l\}
\]
where 
\[
S_{\sigma,o} = \{ (l,j) \in S\; | \; l\in o \wedge (l,j)<\sigma(l,j)\}
\]
where we used the lexicographical order. We have 
\[
\sum_{l\in o} \Big( \sum_{j\in S_{\sigma,l,+}} k_{(l,j)} - \sum_{j \in S_{\sigma,l,-}} k_{\sigma(l,j)}\Big) = \sum_{l\in o, j\in S_{\sigma,l,+}} k_{l,j} - \sum_{l\in o, j\in S_{\sigma,l,-}} k_{\sigma(l,j)}.
\]
Let $l\in o$ and $j\in S_{\sigma, l, +}$. By definition, $\sigma(l,j) = (l',j')$ with $l'>l$. By definition of the orbits, we also have $l'\in o$. Because $\sigma $ is an involution, we have 
\[
(l,j) = \sigma(l',j').
\]
Finally, by definition of $S_{\sigma,l',-}$, we have $j'\in S_{\sigma,l',-}$. In other words, there exists (a unique) couple $(l',j')$ such that $l'\in o$ and $j' \in S_{\sigma,l',-}$ such that
\[
(l,j) = \sigma(l',j').
\]
Conversely, if $(l',j')$ is such that $l'\in o$, $j'\in S_{\sigma,l,-}$ then $(l,j) := \sigma(l',j')$ is such that $l\in o$ and $j\in S_{\sigma,l,+}$. Therefore,
\[
\bigsqcup_{l\in o} S_{\sigma,l,+} = \sigma\Big(\bigsqcup_{l\in o} S_{\sigma,l,-}\Big).
\]

We deduce
\[
\sum_{l\in o} \Big( \sum_{j\in S_{\sigma,l,+}} k_{(l,j)} - \sum_{j \in S_{\sigma,l,-}} k_{\sigma(l,j)}\Big) = 0
\]
and thus $\Sigma_{\sigma,o} \neq \emptyset$ implies
\[
\sum_{l\in o} \xi_l = 0.
\]

Assume now that $\sum_{l\in o} \xi_l = 0$. We write $(E_l)$ the equation 
\[
\sum_{j\in S_{\sigma,l,+}} k_{(l,j)} - \sum_{j \in S_{\sigma,l,-}} k_{\sigma(l,j)} = L \xi_l .
\]
We know that these equations are not independent since 
\[
\sum_{l\in o} (E_l) = 0.
\]
We prove now that at least $\# o -1$ of them are independent. We argue by contradiction. By contradiction, we assume that there exists $(\alpha_l)_{l\in o}$ such that the sequence is not constant and 
\[
\sum_{l\in o} \alpha_l (E_l) = 0.
\]
This would imply that $\sum_{l} \alpha_l \xi_l = 0$ and 
\[
\sum_{l\in o} \sum_{j\in S_{\sigma,j,+}} \alpha_l k_{l,j} - \sum_{l\in o}\sum_{j\in S_{\sigma,l,-}} \alpha_{l} k_{\sigma(l,j)} = 0
\]
for all $\vec k \in (\Z^d_*)^{S_{\sigma,o}}$. This may be rewritten as
\[
\sum_{l\in o}\sum_{j\in S_{\sigma,j,+}} k_{l,j} (\alpha_l - \alpha_{\tilde\sigma(l,j)})=0
\]
for all $\vec k \in (\Z^d)^{S_{\sigma,o}}$.

We deduce that if there exists $j$ such that $\tilde \sigma(l,j) = l'$, then 
\[
\sum_l \alpha_l (E_l) = 0
\]
implies $\alpha_l = \alpha_{l'}$. In other words, for all $l'\in \sigma(\{l\})$, $\alpha_{l'} = \alpha_l$. By induction, we get $\alpha$ is constant on the whole orbit which yields a contradiction. We get indeed that at least $\# o - 1$ equations are independent. We deduce 
\[
\Sigma_{\sigma,o} \sim (\Z^d)^{\# S_{\sigma,o} - \# o +1} 
\]
and thus
\[
\Sigma_\sigma \sim (\Z^d)^{s_\sigma}
\]
with
\[
s_\sigma = \sum_o (\# S_{\sigma,o} - \#o + 1) = \# S_\sigma - R + \# \mathcal O_\sigma 
\]
hence the result.
\end{proof}

\subsection{Estimates}

We estimate the cardinal of $\mathcal A_n$ and $G_{L,A,\vec k}$ for a given $\vec k$. 

\begin{proposition}\label{prop:cardinalA} Let $n\in \N^*$, we have that
\[
\# \mathcal A_n \leq \left \lbrace{\begin{array}{cc}
4^{n-1} & \textrm{ if } N=1\\
(3eN)^{n-1} & \textrm{ otherwise.} \end{array}} \right.
\]
\end{proposition}

\begin{proof} This is a classical computation that we detail here for the seek of completeness. 

Using Polish notation, the trees in $\mathcal A_n$ are sequences of $n(N-1)+1$ zeros and $n$ ones, knowing that the first character is a one and the last a zero. Therefore, it remains to place $n-1$ ones into $n(N-1)+ 1 +n -2 = nN-1$ slots. There are of course extra rules than the ones we mention but this leaves at most
\[
\begin{pmatrix}
nN-1 \\ n-1
\end{pmatrix}
\]
possibilities and thus
\[
\# \mathcal A_n \leq \frac{(nN-1)!}{(n-1)!(n(N-1))!}.
\]

We start with $N=2$. In this case, we have 
\[
\# \mathcal A_n \leq \frac{(2n-1)!}{n! (n-1)!} = \prod_{k=1}^{n-1}\frac{2k}{k} \prod_{k=2}^{n}\frac{2k-1}{k}
\]
which yields the result using that $2k-1 \leq 2k$. 

For general $N$, we have 
\[
(nN-1)! = \prod_{j=1}^{n-1} Nj  \prod_{k=1}^{N-1} \prod_{j=0}^{n-1} (Nj+k).
\]
We deduce that
\[
\frac{(nN-1)!}{(n-1)!} = N^{n-1} \prod_{k=1}^{N-1} \prod_{j=0}^{n-1} (Nj+k).
\]

We also have 
\[
(n(N-1))! = n(N-1) \prod_{j=1}^{n-1} (N-1)j \prod_{k=1}^{N-2} \prod_{j=0}^{n-1} ((N-1)j+k).
\]
We deduce 
\[
\# \mathcal A_n \leq N^{n-1}\Big( \prod_{k=1}^{N-2}\prod_{j=0}^{n-1} \frac{Nj+k}{(N-1)j+k}\Big)\Big( \prod_{j=1}^{n-1} \frac{Nj+N-1}{(N-1)j}\Big) \frac{N-1}{n(N-1)}.
\]
Let 
\[
I = \prod_{j=1}^{n-1} \frac{Nj+N-1}{(N-1)j} .
\]
We have for all $j\in [1,n-1]$,
\[
\frac{Nj+N-1}{(N-1)j} = \frac{N}{N-1} + \frac1{j} \leq \frac{2N-1}{N-1}.
\]
We deduce
\[
I\leq \Big( \frac{2N-1}{N-1}\Big)^{n-1}.
\]
We have of course
\[
\frac{N-1}{n(N-1)} = \frac1{n}.
\]
We set
\[
II = \prod_{k=1}^{N-2}\prod_{j=0}^{n-1} \frac{Nj+k}{(N-1)j+k}.
\]
We have 
\[
\ln II = \sum_{k=1}^{N-2} \sum_{j=0}^{n-1} \ln \Big( 1+ \frac{j}{(N-1)j+k}\Big).
\]
Because $\ln(1+x) \leq x$ for all $x\geq 0$, we have 
\[
\ln II \leq \sum_{j=1}^{n-1} j \sum_{k=1}^{N-2} \frac1{(N-1)j+k}.
\]
We have 
\[
\sum_{k=1}^{N-2} \frac1{(N-1)j+k} \leq \int_{(N-1)j}^{(N-1)j+N-2} \frac{dx}{x} = \ln\Big( \frac{(N-1)j + N-2}{(N-1)j}\Big) = \ln (1+ \frac{N-2}{(N-1)j}) \leq \frac1{j}.
\]
We deduce
\[
\ln II \leq (n-1)
\]
and thus
\[
II \leq e^{n-1}.
\]
Summing up we get
\[
\# \mathcal A_n \leq \Big( eN\frac{2N-1}{N-1}\Big)^{n-1}.
\]
Roughly, we get
\[
\#\mathcal A_n \leq (3eN)^{n-1}.
\]
\end{proof}

\begin{proposition}\label{prop:estimateG} There exists $C = C(N,\Psi)$ such that for all $n_in \N$, for all $\vec k \in (\Z^d_*)^{(N-1)n+1}$ and for all $A \in \mathcal A_n$, for all $t\in \R_+$, we have 
\[
|G_{L,A,\vec k}(t)| \leq C^n t^n \max_{l=1}^{(N-1)n + 1} \an{\frac{k_l}{L}}^{rn} \prod_{j=1}^{(N-1)n+1} |a_{L,k_j}|.
\]
\end{proposition}

\begin{proof} We prove this by induction on the trees. If $A = ()$ then 
\[
|G_{L,(),(k_1)}(t)| = |a_{L,k_1}|.
\]
If $A \in \mathcal A_{n+1}$ with $A= (A_1,\hdots, A_N)$ and $A_j \in \mathcal A_{n_j}$, we have 
\[
|G_{L,A,\vec k}(t)| = \Big| \int_{0}^t e^{i(t-\tau)\omega(R(\vec k)} \Psi(R(\vec k_1),\hdots,R(\vec k_N)) ( G_{L,A_1,\vec k_1}(\tau), \hdots, G_{L,A_N,\vec k_N}(\tau))\Big|.
\]
We use the induction hypothesis to get that
\[
\Big|\prod_{j=1}^N G_{L,A_j,\vec k_j}(\tau)\Big| \leq C^n \tau^n \max_{l=1}^{(N-1)n + 1} \an{\frac{k_l}{L}}^{rn} \prod_{l=\tilde n_j+1}^{\tilde n_j + (N-1)n_j +1}|a_{L,k_l}|.
\]
We deduce
\[
|G_{L,A,\vec k}(t)|\leq C^n \max_{l=1}^{(N-1)n + 1} \an{\frac{k_l}{L}}^{rn} \frac{t^{n+1}}{n+1}|\Psi(R(\vec k_1),\hdots,R(\vec k_N)) | \prod_l |a_{L,k_l}|.
\]
We have that
\[
|\Psi(R(\vec k_1),\hdots,R(\vec k_N)) |\leq C' \max_{j=1}^N \an{R(\vec k_j)}^r.
\]
Since $r\leq 1$, we have 
\[
\an{R(\vec k_j)}^r \leq \sum_{l} \an{\frac{k_l}{L}}^r \leq (n_j(N-1)+1) \max_{l} \an{\frac{k_l}{L}}^r.
\]
Since $n\geq n_j$ for all $j$, we have 
\[
|G_{L,A,\vec k}(t)|\leq C^n N C'\max_{l=1}^{(N-1)n + 1} \an{\frac{k_l}{L}}^{r(n+1)} t^{n+1}\prod_l |a_{L,k_l}|.
\]
Taking $C= C' N$ we get the result.
\end{proof}

We now estimate
\[
F_{L,\sigma} (t) = \frac1{(2\pi L)^{d(N-1)(\sum n_l)/2}} \sum_{A_l \in \mathcal A_{n_l}}  \sum_{\Sigma_\sigma} \prod_{l=1}^R G_{L,A_l,\vec k_l} 
\]

Combining all we have done so far, we get the following proposition.

\begin{proposition} If for some $o \in \mathcal O_\sigma$, we have 
\[
\sum_{l\in o} \xi_l \neq 0
\]
then
\[
F_{L,\sigma}(t) = 0
\]
otherwise, we have the estimate, with $\bar C = C 3eN$ if $N>2$ and $\bar C = 4C$ if $N=2$,
\[
|F_{L,\sigma}(t)| \leq \frac{(\bar C A_L^r t)^{\sum n_l}}{(2\pi L)^{d(R/2 - \#\mathcal O_\sigma})} \|a_{L,k}\|_{\ell^\infty}^{\# S - 2s_\sigma} \|a_{L,k}\|_{\ell^2}^{2s_\sigma}. 
\]
where we recall that $A_L$ is defined as
\[
A_L := \sup \{ \an{\frac{k}{L}} \; |\; a_{L,k} \neq 0\}.
\]
\end{proposition}

\begin{proposition}\label{prop:estExp} If it exists, set $\mathcal  O$ be a maximal partition of $[|1,R|]\cap \N$ such that for all $o \in \mathcal O$, 
\[
\sum_{l\in o} \xi_l = 0.
\]
Then,
\[
\E \Big( \prod_{l=1}^R \hat u_{n_l}^{(i_l)}(t)(\xi_l) \Big) \leq \# \mathfrak S \|a_{L,k}\|_{\ell^\infty \cap \ell^2}^{\# S} \frac{(\bar C A_L^r t)^{\sum n_l}}{(2\pi L)^{d(R/2 - \#\mathcal O)}}.
\]

If such a partition does not exist then 
\[
\E \Big( \prod_{l=1}^R \hat u_{n_l}^{(i_l)}(t)(\xi_l) \Big) = 0
\]
\end{proposition}

\begin{proof} If such a partition does not exist then $F_{L,\sigma} = 0$ for all the $\sigma$.

Otherwise, $F_{L,\sigma} \neq 0$ implies that for all $o \in \mathcal O_\sigma$, 
\[
\sum_{l\in o} \xi_l = 0.
\]
In particular,
\[
\# \mathcal O_\sigma \leq \# \mathcal O,
\]
and thus
\[
\frac1{(2\pi L)^{d(R/2 - \#\mathcal O_\sigma)}} \leq \frac1{(2\pi L)^{d(R/2 - \#\mathcal O)}}.
\]
\end{proof}

\begin{remark} The cardinal of $\mathcal O$ is necessarily smaller than $\frac{R}{2}$ since $\xi_l$ cannot be null. 
\end{remark}

\begin{proposition} Assume that a partition $\sqcup_{ o\in  O}  o = [|1,R|]$ such that for all $ o\in O$,
\[
\sum_{l\in o} \xi_l = 0
\]
 exists. Then, let $ O_1, \hdots,  O_F$ be the maximal partitions of this type, then 
\[
\Big|\E \Big( \prod_{l=1}^R \hat u_{n_l}^{(i_l)}(t)(\xi_l) \Big) - \sum_{J=1}^F\prod_{o \in \mathcal O_J }\E(\prod_{l\in o} \hat u_{n_l}^{(i_l)}(t)(\xi_l))\Big| 
\leq  \# \mathfrak S \|a_{L,k}\|_{\ell^\infty \cap \ell^2}^{\# S} \frac{(\bar C t A_L^r)^{\sum n_l}}{(2\pi L)^{d(R/2 - \#\mathcal O+1)}}.
\]
\end{proposition}

\begin{proof} The $\sigma$s that correspond to the leading order in $L$ of 
\[
\E \Big( \prod_{l=1}^R \hat u_{n_l}(t)(\xi_l) \Big) 
\]
are the ones such that $\mathcal O_\sigma =  O_J$ for some $J$. And thus, this $\sigma$s decompose into involutions $\sigma_{o}$ without fixed points of 
\[
S_{o} = \bigcup_{l\in o} \{l\}\times ([1,n_l(N-1)+1]\cap \N)
\]
with only one orbit.

Conversely, the $\sigma_{o}$ that yield a non-zero contribution to 
\[
\E(\prod_{l\in o} \hat u_{n_l}^{(i_l)}(t)(\xi_l))
\]
have necessarily only one orbit due to the maximality of $O_J$.

Note that if $\mathcal O_\sigma =  O_J$ and $\mathcal O_{\sigma '} = O_{J'}$ then $J\neq J'$ implies $\sigma\neq \sigma'$.
\end{proof}

\begin{corollary} We have 
\[\label{Eq:explicitComp}
\Big|\E \Big( \prod_{l=1}^R \hat u_{n_l}^{(i_l)}(t)(\xi_l) \Big) - \sum_{\mathcal O \in \mathcal P_R}\prod_{\{l,l'\} \in \mathcal O}\E( \hat u_{n_l}^{(i_l)}(t)(\xi_l)\hat u_{n_{l'}}^{(i_{l'})}(t)(\xi_{l'})) \Big| 
\leq  \# \mathfrak S \|a_{L,k}\|_{\ell^\infty \cap \ell^2}^{\# S} \frac{(\bar C t A_L^r)^{\sum n_l}}{(2\pi L)^{d/2})}
\]
where $\mathcal P_R$ is the set of partitions of $[|1,R|]$ whose elements are pairs of $[|1,R|]$. 
\end{corollary}

\begin{remark} In other words, $\mathcal P_R$ is the set of involutions of $[|1,R|]$ without fixed points.
\end{remark}

\section{Application to the Euler equation}

We consider the Euler equation : 
\begin{equation}\label{Euler1}
\left \lbrace{\begin{array}{cc}
\partial_t u_L+ u_L\cdot \nabla u_L = -\nabla p & \textrm{ on } L \T^d\\
\nabla\cdot u_L = 0
\end{array}} \right. .
\end{equation}
Remark that the quantity
\[
\int_{L\T^d} u_L(x) dx
\]
is a priori conserved under the action of the flow of the equation, we chose it null.

Applying the Leray projection defined in Fourier mode as
\[
\widehat{Pv} (\xi) = \hat v (\xi) - \sum_j \xi_j \hat v^{(j)}(\xi) \frac{\xi}{|\xi|^2}
\]
we get that $ u_L= Pu_L$ satisfies
\[
\partial_t u_L + P(u_L\cdot \nabla u_L) = 0.
\]

Therefore, $J$ writes 
\[
J(u,v) = P(u\cdot \nabla v)
\]
and $\Psi$ writes 
\[
\Psi(\xi-\eta, \eta) (X,Y) = \sum_{j=1}^d i\eta_j X^{(j)} Y - \sum_{j,k=1}^d \frac{\xi_k}{|\xi|} i\eta_j X^{(j)} Y^{(k)} \frac{\xi}{|\xi|}
\]
for all $\eta = (\eta_1,\hdots, \eta_d) \in \frac1{L} \Z^d_*$, $\xi  = (\xi_1,\hdots, \xi_d)\in \frac1{L}(\Z^d_*)$ such that $\xi-\eta \neq 0$ and all $X = (X^{(1)},\hdots, X^{(d)}), Y = (Y^{(1)},\hdots, Y^{(d)}) \in \C^d$.

Therefore, we are in the framework afore-mentioned with $r=1$.

Note that the initial datum must satisfy $\nabla \cdot a_L = 0$ which is implied by the condition $\xi \cdot a(\xi) = 0$ for all $\xi$. 

\subsection{Well-posedness in the analytic framework}\label{subsec:wellposedness}

Let $\psi : \R $ be a smooth increasing map with values in $[0,1]$ which is equal to $1$ on $[1,\infty)$ and to $0$ on $(-\infty, 0]$. Set $\varphi (x) = \psi (x+1)$ on $[-1,0]$, $\varphi(x) = 1-\psi(x)$ on $[0,1]$ and $\varphi(x) = 0$ elsewhere. We set for $n\in \Z^d$, and $\xi \in \R^d$,
\[
\varphi_n(\xi) = \prod_{j=1}^d\varphi(\xi_j-n_j).
\]
We get that $\varphi_n$ is smooth, supported in the rectangle
\[
R_n = \prod_{j=1}^d [n_j-1,n_j+1].
\]

We also have that
\[
\sum_{n\in \Z^d} \varphi_n  = Id_{\R^d}.
\]

\begin{definition} Let $\rho >0$, we introduce the space $E_\rho$ induced by the norm
\[
\|f\|_\rho := \sum_n e^{\rho |n|} \|\phi_n * f\|_{L^\infty (\R^d)}
\]
where $\phi_n$ is the inverse Fourier transform of $\varphi_n$.
\end{definition}

\begin{proposition}\label{prop:estimBilin1} There exists $C = C(d,\varphi)$ such that for all $\rho > 0$ and all $f,g\in E_\rho$, we have 
\[
\|fg\|_{\rho} \leq e^{2\rho}  \|f\|_\rho \|g\|_{\rho} .
\]
\end{proposition}

\begin{proof} Let $f_n = \phi_n * f$ and $g_n = \phi_n * g$. We have for all $n$,
\[
\phi_n*(fg) = \phi_n*(\sum_k f_k \sum_l g_l) = \sum_{k,l} \phi_n*(f_kg_l).
\]
We have that $f_k$ is supported in Fourier mode in $R_k$ and $g_l$ is supported in Fourier in $R_l$ hence $f_kg_l$ is supported in the rectangle
\[
\prod_{j=1}^d [k_j+l_j-2, k_j+l_j+2].
\]
Therefore,
\[
\phi_n * (f_kg_l) \neq 0
\]
implies that for all $j=1,\hdots, d$,
\[
[n_j-1,n_j+1] \cap [k_j+l_j-2, k_j+l_j+2] 
\]
is not of null Lebesgue measure. In other words, $n_j-1$ has to be strictly smaller that $k_j+l_j+2$ and $n_j+1$ has to be strictly greater than $k_j+l_j-2$, that is
\[
n_j \in [| k_j+l_j-2, k_j +l_j +2|].
\]
In particular, $|n_j|\leq |k_j| + |l_j| + 2$ and there are $5^d$ tuples $n$ that satisfy this thus
\[
\phi_n *(fg) = \sum_{|n_j| \leq |k_j|+|l_j|+2 } \phi_n*(f_k g_l).
\]
Therefore, we have 
\[
\| \phi_n*(fg)\|_{L^\infty} \leq \sum_{|n|\leq |k|+|l|+2} \|\phi_n*(f_kg_l)\|_{L^\infty}.
\]
Since 
\[
\phi_n (x) = e^{in\cdot x} \phi_0(x),
\]
we get that $\phi_n$ belongs to $L^1$ its norm is uniformly bounded in $n$, we get
\[
\| \phi_n*(fg)\|_{L^\infty} \lesssim \sum_{|n|\leq |k|+|l|+2} 1_n(k,l) \|(f_kg_l)\|_{L^\infty}
\]
where $1_n(k,l)$ equals $1$ if $n_j \in [k_j+l_j-2,k_j+l_j+2]$ for all $j$ and $0$ otherwise.

We have that $L^\infty$ is an algebra and thus
\begin{equation}\label{StartAgain}
\| \phi_n*(fg)\|_{L^\infty} \lesssim \sum_{|n|\leq |k|+|l|+2}1_n(k,l) \|f_k\|_{L^\infty}\|g_l\|_{L^\infty}.
\end{equation}
We sum over $n$ and get
\[
\|fg\|_{\rho} \lesssim \sum_{k,l} \|f_k\|_{L^\infty}\|g_l\|_{L^\infty} \sum_{|n|\leq |k|+|l|+2}1_n(k,l) e^{\rho' |n|}.
\]
We deduce
\[
\|fg\|_{\rho} \lesssim e^{2\rho} \sum_{k,l} e^{\rho |k|}\|f_k\|_{L^\infty}e^{\rho |l|}\|g_l\|_{L^\infty}.
\]
Hence the result.
\end{proof}

\begin{proposition} Set $\chi_1$ a smooth map that is equal to $1$ on $\{|\xi|\geq 1\}$ and null on $\{|\xi|\leq 1/2\}$. We identify $\chi_1$ and the Fourier multiplier by $\chi_1$. There exists $C = C(d,\varphi)$ such that for all $\rho \geq 0$ and all $f \in E_\rho$,
\[
\|P\chi_1 f\|_\rho \leq C \|f\|_\rho.
\]
\end{proposition}

\begin{proof} Because $P$ and $\chi_1$ commute with Fourier multipliers, we have that
\[
\phi_n *(P\chi_1 f) = P\chi_1 (\phi_n *f).
\]
We have that $P \chi_1 $ acts as a smooth Fourier multiplier on $\phi_n * f$. We set $\chi$ a $\C^\infty$ map that is non-negative, equal to $1$ on $[-1,1]^d$ and null outside $[-3/2,3/2]^d$.  We write also $P$ the kernel of the Leray projection and 
\[
P_n(\xi) = \chi(\xi-n) \chi_1(\xi) P(\xi).
\]
We have that the inverse Fourier transform of $P_n$ is in $L^1$ and that its norm is less than
\[
\| P_n\|_{H^s}
\]
for $s> d/2$. We get that
\[
\|P_n\|_{H^s} \leq \|\chi(\cdot -n)\|_{H^s} \|\chi_1 P\|_{W^{s,\infty}}
\]
we deduce that the $L^1$ norm of the inverse Fourier transform of $P_n$ is uniformly bounded on $n$ and thus
\[
\|\phi_n*(Pf)\|_{L^\infty} = \|P_n \phi_n*f\|_{L^\infty} \leq \|f_n\|_{L^\infty}.
\] 
\end{proof} 

\begin{proposition} There exists $C = C(d,\varphi)$ such that for all $\rho > \rho' \geq 0$, and for all $f\in E_{\rho}$, we have 
\[
\|\nabla f\|_{\rho'} \leq C e^{\rho'} (\rho-\rho')^{-1} \|f\|_{\rho'} .
\]
\end{proposition}

\begin{proof} We have 
\[
(\nabla f)_n = (\nabla \phi_n)*f.
\]

For the usual support considerations, we have 
\[
(\nabla f)_n = \sum_{|n'-n|\leq 1} (\nabla \phi_n) * f_{n'}.
\]
Indeed, we have 
\[
(\nabla f)_n = \sum_{n'} (\nabla \phi_n)*f_{n'}.
\]
What is more, $\nabla \phi_n$ is supported in Fourier modes in $R_n$ and $f_{n'}$ is supported in Fourier modes in $R_{n'}$. If $R_{n}\cap R_{n'}$ is not negligible, then for all $j$, we have 
\[
n_j \in (n'_j-1,n'_j+1)
\]
that is for all $j$, $|n_j-n'_j|\leq 1$.

We have that
\[
 \nabla \phi_n = in e^{ix\cdot n}\phi_0 + e^{ix \cdot n}\nabla \phi_0
\]
and thus
\[
|\nabla \phi_n \|_{L^1} \lesssim (|n|+1).
\]
We deduce 
\[
\|(\nabla f)_n \|_{L^\infty} \leq \sum_{|n'-n|\leq 1} (|n|+1)\|f_{n'}\|_{L^\infty}.
\]
We sum on $n$ and get the result using that
\[
e^{\rho' |n|} |n| \leq \frac1{\rho-\rho'} e^{\rho |n|}.
\]
\end{proof}

\begin{proposition}  There exists $C = C(d,\varphi)$ such that for all $u = Pu \in E_\rho$ and all $v \in E_\rho$ such that $\nabla v \in E_\rho$, we have 
\[
\|P(u\cdot \nabla v)\|_\rho \leq C e^{2\rho} \|u\|_\rho(\|v\|_\rho + \|\nabla v\|_\rho)
\]
and 
\[
\| \nabla P (u\cdot \nabla v) \|_\rho \leq C e^{2\rho} (\|u\|_\rho \|\nabla \otimes \nabla v\|_\rho + \|\nabla u \|_\rho \|\nabla v\|_\rho + \|u\|_\rho \|\nabla v\|_\rho).
\]
\end{proposition}

\begin{proof} We write $P = P\chi_1 + P(1-\chi_1)$. As we have already seen, $P\chi_1$ is smooth in Fourier modes which, combined with the estimates on the product of two maps is sufficient to conclude that 
\[
\|P\chi_1 (u\cdot \nabla v)\|_\rho \lesssim e^{2\rho}\|u\|_\rho\|\nabla v\|_\rho,\quad 
\| \nabla P\chi_1 (u\cdot \nabla v) \|_\rho \leq C e^{2\rho} (\|u\|_\rho \|\nabla \otimes \nabla v\|_\rho + \|\nabla u \|_\rho \|\nabla v\|_\rho ).
\]
Now $\nabla P (1-\chi_1)$ is not a $\mathcal C^\infty$ Fourier multiplier but it is compactly supported and its behaviour at $0$ is sufficiently smooth. Indeed, we have, for $v$ in the Schwartz class,
\[
\nabla P (1-\chi_1) v(x) = \int d\xi (1-\chi_1(\xi)) \Big(\xi \otimes \hat v(\xi) - \sum_j \xi_j \hat v^{(j)}(\xi)  \frac{\xi \otimes \xi}{|\xi|^2} \Big) e^{i\xi (x-y)}.
\]
We use the inverse Fourier transform and get
\[
\nabla P (1-\chi_1)v(x) = \int dy v(y) \otimes F(x-y) - \sum_j \int dy v^{(j)}(y) F_j(x-y)
\]
where
\[
F(z) = \int d\xi (1-\chi_1(\xi)) \xi e^{i z} , \quad F_j(z) = \int d\xi \frac{\xi \otimes \xi }{|\xi|^2} \xi_j (1-\chi_1(\xi)) e^{i\xi z}.
\]
These two functions are well-defined because 
\[
\xi \mapsto (1-\chi_1(\xi)) \xi, \quad \xi \mapsto \frac{\xi \otimes \xi }{|\xi|^2} \xi_j (1-\chi_1(\xi))
\]
are continuous and compactly supported.

We prove that $F$ and $F_j$ are in $L^1$. For $F$, this is the case because $\xi \mapsto (1-\chi_1(\xi)) \xi$ is $\mathcal C^\infty$ and compactly supported. For $F_j$, we have that
\[
\|F_j\|_{L^1} \lesssim \|\hat F_j\|_{H^{d/2 +\eta}}
\]
for any $\eta > 0$. We have that $\hat F_j$ is smooth outside of $0$ and compactly supported. At $0$, it behaves like
\[
\frac{(\xi\times \xi)\xi_j}{|\xi|^2}
\]
and thus
\[
|\an{\nabla_\xi}^\alpha \hat F_j (\xi)| \sim c |\xi|^{1-\alpha}.
\]
We have that $\xi \mapsto |\xi|^{1-\alpha}$ belongs to $L^2$ if $1-\alpha > -\frac{d}2$. In particular, we have that $\hat F_j$ belongs to $H^\alpha$ for $\alpha \in (\frac{d}{2}, \frac{d}{2} +1)$. Therefore, $F_j$ belongs to $L^1$. By duality, we get that for $v\in L^\infty$, we have 
\[
\nabla P (1-\chi_1)v(x) = \int dy v(y) \otimes F(x-y) - \sum_j \int dy v^{(j)}(y) F_j(x-y)
\]
and that
\[
\|\nabla P (1-\chi_1)v(x) \| \leq (\|F\|_{L^1} + \|F_j\|_{L^1}) \|v\|_{L^\infty}.
\]

For $P(1-\chi_1)(u\cdot \nabla v)$, we use that $u=Pu$ and thus $u\cdot \nabla v = \partial_j (u^{(j)} v)$ and then we use the same arguments as for $\nabla P (1-\chi_1)$ to get
\[
\|  P (1-\chi_1) (u\cdot \nabla v)\|_\rho \lesssim \|u\|_\rho \|v\|_\rho.
\]
\end{proof}

We now prove bilinear estimates such that we can make the Picard expansion converge. The idea is to render explicit the Cauchy-Kowalevskaia abstract theorem.

\begin{definition} Let $\rho_0>0$, $\beta \in (0,1)$ and $\theta >0$. We set for all $\rho \in (0,\rho_0)$, $\theta(\rho) = \theta (\rho_0-\rho)$. We define $M(\rho_0, \beta, \theta)$ the space induced by the norm
\[
\|u\|_{\rho_0,\beta, \theta} = \sup_{0<\rho < \rho_0} \sup_{0\leq t< \theta(\rho)}\Big( \|u(t)\|_{\rho} + \|\nabla u(t)\| (\theta(\rho)-t)^\beta\Big).
\]
\end{definition}

\begin{proposition} Let $\rho_0 >0$, and $\beta \in (0,1)$. There exists $C = C(d,\varphi, \rho_0,\beta)$ such that for all $u,v\in  M(\rho_0,\beta,\theta)$ such that $u=Pu$, we have 
\[
\big\| \int_{0}^t P(u(\tau)\cdot \nabla v (\tau))d\tau \big\|_{\rho_0,\beta,\theta} \leq C \theta^{1+\beta} \|u\|_{\rho_0,\beta,\theta} \|v\|_{\rho_0,\beta,\theta}.
\]
Besides 
\[
P \int_{0}^t P(u(\tau)\cdot \nabla v (\tau))d\tau = \int_{0}^t P(u(\tau)\cdot \nabla v (\tau))d\tau.
\]
\end{proposition}

\begin{proof} For the sake of this proof, we set
\[
A (t) \int_{0}^t P(u(\tau)\cdot \nabla v (\tau))d\tau, \quad \|\cdot\| = \|\cdot \|_{\rho_0,\beta,\theta}.
\]

Let $\rho\in (0,\rho_0)$ and $t\in (0,\theta(\rho))$.  We have 
\[
\|A(t)\|_\rho \leq \int_{0}^t \|P(u\cdot \nabla v)(\tau)\|_\rho d\tau.
\]
We have 
\[
\|A(t)\|_\rho \leq \int_{0}^t e^{2\rho} \|u(\tau)\|_{\rho}(\|v(\tau)\|_\rho + \|\nabla v(\tau)\|_{\rho})d\tau.
\]
We use that $\|\nabla v(\tau)\|_\rho \leq \|v\| (\theta(\rho) - \tau)^{-\beta}$ and that
\[
\int_{0}^t (\theta(\rho) - \tau)^{-\beta}d\tau \leq \frac1{1-\beta}(\theta(\rho))^{1-\beta}
\]
to get 
\[
\|A(t)\|_\rho \leq C e^{2\rho}(\frac{(\theta \rho_0)^{1-\beta}}{1-\beta} + \theta) \|u\|\,\|v\|.
\]

We have that
\[
\|\nabla A(t)\| \leq C e^{2\rho} \int_{0}^t (\|\nabla u(\tau)\|_\rho \| \nabla v(\tau)\|_\rho + \|u(\tau)\|_\rho\| \nabla\otimes \nabla v(\tau)\|_\rho + \|u(\tau)\|_\rho \|v(\tau)\|_\rho) d\tau.
\]
We estimate
\[
I(t)= e^{2 \rho}\int_{0}^t  \|\nabla u(\tau)\|_\rho \|\nabla v(\tau)\|_\rho d\tau.
\]
We use that $u$ and $v$ belong to $M(\rho_0,\beta,\theta)$ to get
\[
I(t)\leq e^{3 \rho}\|u\|\,\|v\|\int_{0^t}  (\theta(\rho) - \tau)^{-2\beta} d\tau.
\]
We use that $1-2\beta \geq -\beta$ to get
\[
I(t)\leq\left \lbrace{\begin{array}{cc}
\frac{ e^{3 \rho}}{1-2\beta}(\theta \rho_0)^{1-2\beta}\|u\|\,\|v\| (\theta(\rho) - t)^{-\beta} & \textrm{ if } \beta \neq \frac12 \\
\frac{2}{e} \sqrt{\theta \rho_0}\frac1{\sqrt{\theta(\rho) - t}}\|u\|\,\|v\| \textrm{ otherwise.} \end{array}}\right.
\]

We estimate
\[
II(t) = e^{2\rho }\int_{0}^t \|u(\tau)\|_\rho \|\nabla\otimes \nabla v(\tau)\|_\rho d\tau.
\]

We set 
\[
\rho(\tau) = \rho_0 -  \frac{\theta(\rho) + \tau}{2\theta} .
\]
By definition $\rho(\tau)< \rho_0$.

Because $\tau < \theta(\rho)$, we have 
\[
\rho(\tau) > \rho_0 - \frac{\theta(\rho)}{\theta} = \rho.
\]

We deduce 
\[
II(t) \leq e^{3\rho }\|u\|\int_{0}^t  \frac1{\rho(\tau) -\rho}\| \nabla v(\tau)\|_{\rho (\tau)} d\tau.
\]

We also have that 
\[
\theta(\rho(\tau)) =\frac{ \theta(\rho) +\tau}{2} 
\]
and thus
\[
\theta(\rho(\tau) )-\tau  =  \frac{\theta(\rho) - \tau}{2} >0.
\]

We deduce that since $u,v\in M(\rho_0, \beta,\theta)$,
\[
II(t) \leq e^{3\rho} \|u\|\,\|v\|\int_{0}^t (\theta(\rho(\tau))-\tau)^{-\beta}(\rho(\tau) - \rho)^{-1}d\tau 
\]
and thus
\[
II(t) \leq 2 e^{3\rho} \|u\|\,\|v\|\int_{0}^t (\theta(\rho)-\tau)^{-\beta}(\rho(\tau) - \rho)^{-1}d\tau .
\]

By definition, we have
\[
\rho(\tau) - \rho = \rho_0 - \rho -  \frac{\theta(\rho) + \tau}{2\theta} = \frac{\theta(\rho) - \tau}{2\theta},
\]
We get
\[
II(t) \leq 4\theta e^{3\rho} \|u\|\,\|v\|\int_{0}^t (\theta(\rho)-\tau)^{-\beta-1}d\tau.
\]
We deduce 
\[
II(t) \leq 4\theta^{1-\beta}\rho_0^{-\beta} e^{3\rho} \|u\|\, \|v\| (\theta(\rho) - t)^{-\beta}.
\]
Finally,
\[
III(t) := \int_{0}^t\|u(\tau)\|_\rho \|v(\tau)\|_\rho d\tau \leq \|u\|\,\|v\| t \lesssim \|u\|\,\|v\| (\theta(\rho) - t)^{-\beta} \theta^{1+\beta}.
\]
\end{proof}

\begin{proposition}\label{prop:wellPosedness} Let $\rho_0 >0$ and $\theta \geq 1$. Let $u_0 \in E_{\rho_0}$ and define by induction on $n$,
\[
u_{n+1} = \sum_{n_1+n_2 = n} \int_{0}^t P(u_{n_1}(\tau)\cdot u_{n_2}(\tau))d \tau.
\]
There exists $C = C(d,\varphi, \beta,\rho_0)$ such that for all $n$, we have 
\[
\|u_n\|_{\rho_0,\beta,\theta} \leq \theta^{(1+\beta)n} C^n \|u_0\|_{\rho_0}^{n+1}
\]
such that the series $u_n$ converge in $M(\rho_0,\beta,\theta)$ if $\|u_0\|_{\rho_0} < \theta^{\beta - 1}C^{-1}$ towards the unique solution to the Euler equation with initial datum $u_0$.
\end{proposition}

\begin{proof} We can check by induction on $n$ that
\[
\|u_n\|_{\rho_0,\beta,\theta} \leq \theta^{(1+\beta)n} C^nc_n\|u_0\|_{\rho_0}^{n+1}
\]
where $C$ is the constant of the previous proposition and where $c_n = \# \mathcal A_ns$ are the Catalan numbers. We then use that $c_n\leq 4^n$.
\end{proof}

\subsection{Estimations on the norm of the initial datum}

Here, we set $a_{L,k} = \varepsilon(L) a(\frac{k}{L})$ where $a$ is a bounded, compactly supported function and where $\varepsilon (L) = \mathcal O((\ln L)^{-1/2})$ for some $\varepsilon>0$.

\begin{proposition}\label{prop:estDI} There exists $C = C(d,a,\rho_0,\varphi)$ and $c= c(d,a,\rho_0,\varphi)>0$ such that for all $L \geq e^2$, and all $R\geq \sqrt{\ln L}\varepsilon(L) C$,
\[
\mathbb P(\|a_L\|_{\rho_0} \geq R) \leq e^{-cR^2/\varepsilon(L)^2}.
\]
\end{proposition}

\begin{proof} Let $p\geq 2$, we estimate 
\[
\E_p^p := \E (\|a_L\|_{\rho_0}^p) .
\]
We have 
\[
\E_p \leq \sum_n e^{\rho |n|} \|(a_L)_n\|_{L^p(\Omega,L^\infty)}.
\]
Since $a_L$ is $L$ periodic, so is $(a_L)_n = \phi_n*a_L$.

We deduce that
\[
\|(a_L)_n\|_{L^\infty} \leq \sup_{X} \|\chi_X (a_L)_n\|_{L^\infty}
\]
where the the supremum is taken over the $X\in \Z^d$ such that $|X|\leq L$ and such that $\chi_X = \chi_0(\cdot - X) = \prod_j \chi(\cdot_j - X_j)$ where $\chi$ is a smooth function supported on $[-2,2]$ and equal to $1$ on $[-1,1]$. By the Sobolev injection, for $s\in (\frac{d}{2},\infty)\cap \N$, we have 
\[
\|\chi_X (a_L)_n\|_{L^\infty}\lesssim \|\chi_X (a_L)_n\|_{H^s} \lesssim \sup_{Y,s'\leq s} \|\chi_Y (a_L^{s'})_n\|_{L^2}
\]
where $a_L^{s'} = \nabla^{\otimes s'} a_L$. We deduce
\[
\|(a_L)_n\|_{L^p(\Omega,L^\infty)} \leq \|\sup_{Y,s'}  \|\chi_Y (a_L^{s'})_n\|_{L^2}\|_{L^p} \leq \sup_{s'\leq s} \Big(\sum_{Y}  \|\chi_Y (a_L^{s'})_n\|_{L^p(\Omega,L^2)}^p\Big)^{1/p}.
\]
The sum on $Y$ is for $Y \in [|-L,L|]^d$ hence we sum on $(2L+1)^d$ factors.

By Minskowski's inequality, since $p\geq 2$,
\[
\|\chi_Y (a_L^{s'})_n\|_{L^p(\Omega,L^2)} \leq \|\chi_Y (a_L^{s'})_n\|_{L^2(\R^d, L^p(\Omega))}.
\]
The law of $a_L$ is invariant under the action of space translations, hence so is the law of $(a_L^{s'})_n$ and thus
\[
 \|\chi_Y (a_L^{s'})_n\|_{L^2(\R^d, L^p(\Omega))} =  \|\chi_0 (a_L^{s'})_n\|_{L^2(\R^d, L^p(\Omega))}.
\]
We get
\[
\|(a_L)_n\|_{L^p(\Omega,L^\infty)} \lesssim  L^{d/p}\sup_{s'\leq s} \|\chi_0 (a_L^{s'})_n\|_{L^2(\R^d,L^p(\Omega))}.
\]
Still using the invariance under space translations, we get
\[
\|(a_L)_n\|_{L^p(\Omega,L^\infty)} \lesssim  L^{d/p} \|\chi_0\|_{L^2(\R^d)} \sup_{s'\leq s}\|(a_L^{s'})_n(0)\|_{L^p(\Omega)}.
\]
We use that $(a_L^{s'})_n(0)$ is a Gaussian and that $\chi_0$ does not depend on $L$ to get
\[
\|(a_L)_n\|_{L^p(\Omega,L^\infty)} \lesssim  L^{d/p} \sqrt p \sup_{s'\leq s}\|(a_L^{s'})_n(0)\|_{L^2(\Omega)}.
\]
We note that
\[
\|(a_L^{s'})_n(0)\|_{L^2(\Omega)}^2 = \frac1{2\pi L }\sum_k \Big|\frac{k}{L}\Big|^{2s'}|a_{L,k}|^2 \varphi_n(k/L)^2 \lesssim \int \an{\xi}^{2s} \varphi_n^2(\xi) \varepsilon(L)^2 |a(\xi)|^2.
\]
Summing over $n$ and using Cauchy-Schwarz inequality, we get
\[
\E_p \lesssim \sum_n e^{\rho_0 |n|}\|(a_L)_n\|_{L^p(\Omega,L^\infty)} \lesssim L^{d/p}\sqrt p \Big( \sum_n e^{2\rho_0 |n|} \an{n}^{4s} \|(a_L^s)_n(0)\|_{L^2(\Omega)}^2\Big)^{1/2}.
\]
We get
\[
\E_p \lesssim L^{d/p}\varepsilon(L) \sqrt p \Big( \int d\xi \an{\xi}^{4s} e^{2\rho_0 |\xi|} |a(\xi)|^2\Big)^{1/2}
\]
In other words, there exists $C(a,\rho_0,d)$ such that for all $p\geq 2$, and all $L$
\[
\E_p \leq C L^{d/p}\varepsilon(L) \sqrt p.
\]

By Markov's inequality, we deduce that for all $p,R,L$, we have 
\[
\mathbb P( \|a_L\|_{\rho_0} \geq R) \leq R^{-p} C^p L^d\varepsilon(L)^p p^{p/2},
\]
that is
\[
\mathbb P( \|a_L\|_{\rho_0} \geq R) \leq \Big(\frac{CL^{d/p}\varepsilon(L) \sqrt p}{R}\Big)^{p}.
\]
We set $p = \frac{R^2}{C^2e^{d+1} \varepsilon(L)^2}$ taking $R$ as in the hypothesis with a big enough constant, we get
\[
p \geq \ln L \geq 2.
\]
We get
\[
\mathbb P( \|a_L\|_{\rho_0} \geq R) \leq (L^{d/p}e^{-(d+1)})^{p}.
\]
We have $L^{d/p} = e^{d\ln L/p} \leq e^d$, hence
\[
\mathbb P( \|a_L\|_{\rho_0} \geq R) \leq e^{-p} = e^{-R^2/ (\varepsilon(L)^2e^{d+1} C^2)}
\]
hence the result.
\end{proof}

\subsection{Conclusion}

Let $\rho_0>0, \beta \in (0,1)$ and $\theta \geq 1$ and set $A = A(\theta) = \frac{\theta^{-\beta - 1}}{2C}$ where $C$ is the constant mentioned in Proposition \ref{prop:wellPosedness}. Now set
\[
\mathcal E_L = \mathcal E_L(\rho_0,\beta,\theta) = \{\|a_L\|_{\rho_0}\leq A\}.
\]

If $\varepsilon(L) = o((\ln L)^{-1/2})$ then for $L$ big enough, we have that $A$ is big enough to get
\[
\mathbb P(\mathcal E_L) \geq 1 - e^{-cA^2\varepsilon(L)^{-2}}.
\]
If $\varepsilon(L) = O((\ln L)^{-1/2})$, for $A$ to be big enough to get the above inequality, one needs $\theta$ to be small enough. We assume then that $\theta$ is small enough to get the estimate on the measure of $\mathcal E_L$.

We also have that for all $u_0 \in \mathcal E_L$. The solution $u$ to the Euler equation exists and is unique in $M(\rho_0,\beta, \theta)$ and satisfies that for all $n\in \N$, $t< \theta$,
\[
\|u_n(t)\|_0 \leq 2^{-n}A.
\]
Therefore, for the rest of this subsection, we fix $R\in \N^*$, $(\xi_1,\hdots, \xi_R) \in \frac1{L}(\Z^d_*)^R$, $(i_1,\hdots,i_R) \in ([1,d]\cap \N)^R$ and finally
\[
I = \E \Big(1_{\mathcal E_L} \prod_{l=1}^R \hat u^{i_l}(\xi_l)(t)\Big).
\]

We assume that $\varepsilon^{-1}(L) \lesssim L^\alpha$ for some $\alpha \geq 0$ and we set $M = M(L)$ such that for $L$ big enough,
\[
2\ln 2 \frac{M(L)}{\ln L} > \frac{(R+1)d}{2} +R\alpha
\]
and 
\[
C(a,\theta) \sqrt{2R}\sqrt{M(L) + 1} \varepsilon(L) \leq \frac12
\]
for $C(a,\theta)$ a constant that appears in the proof of Lemma \ref{lem:genEst} and depends only on $a$ and $\theta$. For such a $M(L)$ to exists, this requires that 
\[
\varepsilon(L)\sqrt{\ln L} \leq c(a,\theta,R)
\]
for a constant $c(a,\theta, R)>0$ that depends only on $a$, $\theta$ and $R$, which is small enough.

We now write 
\[
u (t) = \mathcal U_M(t) + \mathcal R_M(t)
\]
with
\[
\mathcal U_M (t) = \sum_{n=0}^M u_n(t), \quad \mathcal R_M(t) = \sum_{n>M} u_n(t).
\]
We get
\[
\|\mathcal U_M(t)\|_0 \leq 2 A, \quad \|\mathcal R_M (t)\|_0 \leq 2^{-M}A.
\]

\begin{lemma}\label{lem:removingRM} We have 
\[
I = \E\Big(1_{\mathcal E_L} \prod_{l=1}^R \hat{\mathcal U_M}^{(i_l)}(t) (\xi_l)\Big) + \mathcal O_{d,\varphi, \rho_0,\beta,\theta,R}( \varepsilon(L)^R L^{-d/2}).
\]
\end{lemma}

\begin{proof}
Take $v\in E_0$ and $2\pi L$ periodic. We have that
\[
v = \sum_n \phi_n * v
\]
and that for all $n$, $\phi_n * v$ is $2\pi L$ periodic. Therefore
\[
\hat v  = \sum_n \widehat{\phi_n * v}.
\]
We deduce that
\[
\|\hat v\|_{L^\infty} \leq \sum_n \|\widehat{\phi_n * v}\|_{L^\infty}.
\]
We recall that in the torus $L\T^d$, we define
\[
\hat w (\xi) := \int_{L\T^d} w(x) \frac{e^{-ix\xi}}{(2\pi L)^{d/2}}dx
\]
and thus
\[
\|\hat w\|_{L^\infty} \leq (2\pi L)^{d/2} \|w\|_{L^\infty}.
\]
Therefore,
\[
\|\hat v\|_{L^\infty} \leq \sum_n (2\pi L)^{d/2}\|\phi_n*v\|_{L^\infty} = (2\pi L)^{d/2} \|v\|_0.
\]
We deduce that on $\mathcal E_L$,
\[
\|\hat u(t)\|_{L^\infty} \leq (2\pi L)^{d/2} 2A, \quad \|\hat{\mathcal U}_M(t)\|_{L^\infty} \leq (2\pi L)^{d/2} 2A, \quad \|\hat{\mathcal R}_M(t)\|_{L^\infty}\leq (2\pi L)^{d/2} 2^{-M}A.
\]
We deduce that
\[
I = \E\Big(1_{\mathcal E_L} \prod_{l=1}^R \hat{\mathcal U_M}^{(i_l)}(t) (\xi_l)\Big) + \mathcal O_{d}( 2^{-M} R (2A)^R L^{(dR)/2}).
\]
We get the result since 
\[
2^{-M} L^{(dR)/2} = O(\varepsilon(L)^R L^{-d/2}).
\]
Indeed, we have 
\[
2^{-M} L^{d/2(R+1)} \varepsilon(L)^{-R} \leq e^{-(\ln L) (\ln 2 M(L)/\ln L - (R+1)d/2 -R\alpha)}.
\]
For $L\gg1$, we have 
\[
2\ln 2 M(L)/\ln L - (R+1)d/2 -R\alpha >0
\]
which ensures the result.
\end{proof}

\begin{lemma}\label{lem:genEst} We have 
\[
 \E \Big( \prod_{l=1}^R \hat {\mathcal U}_M^{(i_l)}(\xi_l)(t)\Big) = \mathcal O_{d,\varphi, \rho_0,\beta,\theta,R,a,\varepsilon} (\varepsilon(L)^R).
\]
\end{lemma}

\begin{proof} Set
\[
II :=  \E \Big( \prod_{l=1}^R \hat {\mathcal U}_M^{(i_l)}(\xi_l)(t)\Big).
\]
We have 
\[
II = \sum_{n_l \leq M} \E\Big( \prod_{l=1}^R \hat u_{n_l}^{(i_l)}(\xi_l)(t)\Big).
\]
By Proposition \ref{prop:estExp}, we have 
\[
\Big|\E\Big( \prod_{l=1}^R \hat u_{n_l}^{(i_l)}(\xi_l)(t)\Big) \Big| \leq \#\mathfrak S \|a_{L,k}\|_{\ell^2 \cap \ell^\infty}^{\# S} (t \bar CA_L)^{\sum n_l} (2\pi L)^{d(R/2 - \# O)}
\]
where 
\[
S = \{ (l,k)\; |\; l\in [|1,R|],\; k\in [|1,n_j+1|]\}
\]
where $\mathfrak S$ is the set of involutions of $S$ without fixed points, where 
\[
A_L = \sup \{ \an{k/L} \;|\; a_{k,L} \neq 0\},
\]
and where $O$ is a maximal partition of $[1,R]\cap \N$ such that for all $o\in O$,
\[
\sum_{l\in o} \xi_l = 0.
\]
We have that $\# S = \sum n_l + R$ and 
\[
a_{L,k} =  \varepsilon(L) a(k/L)
\]
and thus
\[
A_L \leq  A_\infty := sup\{\an{\xi} \; |\; a(\xi)\neq 0\}
\]
and 
\[
 \|a_{L,k}\|_{\ell^2 \cap \ell^\infty}^{\# S} ( 2 t \bar CA_L)^{\sum n_l} \leq (C(a,\theta)\varepsilon(L))^{\sum n_l} (C'(a)\varepsilon(L))^R
 \]
where $C(a,\theta) =  2\bar C A_\infty^{1+d/2} \theta \|a\|_{L^\infty}^2$ and $C'(a) = A_\infty^{1+d/2}\|a\|_{L^\infty}^2$.  

Since $O$ cannot contain singletons, we have 
\[
\# O \leq R/2, \quad d(R/2 - \# O)\geq 0.
\]

We deduce
\[
\Big|\E\Big( \prod_{l=1}^R \hat u_{n_l}^{i_l}(\xi_l)(t)\Big) \Big| \leq \# \mathfrak S (C(a,\theta)  \varepsilon(L))^{\sum n_l} (C'(a)\varepsilon(L))^R .
\]

We have that
\[
\#\mathfrak S = \begin{pmatrix}
\# S \\ \#S/2
\end{pmatrix} (\# S/2) !\leq 2^{\# S}(\# S/2)! \leq 2^R R! (2\# S)^{\frac12\sum n_l}.
\]

Since $\# S \leq (M+1)R$, we get
\[
\Big|\E\Big( \prod_{l=1}^R \hat u_{n_l}^{i_l}(\xi_l)(t)\Big) \Big| \lesssim_R  (C(a,\theta)  \sqrt{2 R (M+1)} \varepsilon(L))^{\sum n_l} (C'a)\varepsilon(L))^R .
\]
For $L$ big enough, $C(a,\theta) \alpha R(M(L) + 1)\varepsilon(L) \leq  \frac12$. We deduce that for $L$ big enough
\[
II \lesssim_{a,\theta,R,\varepsilon} \varepsilon(L)^R \sum_{n_l} 2^{-\sum n_l -R } = 1.
\]
Hence the result.
\end{proof}

\begin{lemma}\label{lem:removeEL} We have 
\[
\E\Big( (1-1_{\mathcal E_L}) \prod_{l=1}^R \hat{\mathcal U_M}^{(i_l)}(\xi_l) \Big) = \mathcal O_{d,\varphi, \rho_0,\beta,\theta,R,a,\varepsilon} (\varepsilon(L)^R L^{-d/2}).
\]
\end{lemma}

\begin{proof} We simply use Cauchy-Schwarz inequality to get
\[
\E\Big( (1-1_{\mathcal E_L}) \prod_{l=1}^R \hat{\mathcal U_M^{(i_l)}}(\xi_l) \Big)^2 \leq \E (1-1_{\mathcal E_L}) \E \Big(\prod_{l=1}^{2R} \hat{\mathcal U_M^{(i_l)}}(\xi_l) \Big)
\]
with $\xi_{R+l} = \xi_l$. We use Lemma \ref{lem:genEst} with $R$ replaced by $2R$ to get
\[
\E \Big(\prod_{l=1}^{2R} \hat{\mathcal U_M}(\xi_l) \Big) = \mathcal O_{d,\varphi, \rho_0,\beta,\theta,R,a, \delta,\varepsilon} (1).
\]
We also have that
\[
\E (1-1_{\mathcal E_L}) = \mathbb P (\mathcal E_L^c) \leq e^{-cA^2/\varepsilon( L)^2}. 
\]
Since
\[
\varepsilon(L)^{-2}\geq \frac{\ln L}{c(R)^2},
\]
we get that 
\[
\sqrt{\E (1-1_{\mathcal E_L})}\leq L^{-cA^2/2c(a,\theta,R)^2}
\]
which is a $O (\varepsilon(L)^R L^{-d/2})$ for $c(a,\theta, R)$ small enough and concludes the proof.
\end{proof}

\begin{lemma}\label{lem:explicitComp} We have 
\[
 \E \Big( \prod_{l=1}^R \hat {\mathcal U}_M^{(i_l)}(\xi_l)(t)\Big) = \sum_{\mathcal O \in \mathcal P_R}\prod_{\{l,l'\} \in \mathcal O}\E( \hat {\mathcal U}_{M}(t)(\xi_l)^{(i_l)}\hat{\mathcal U}_{M}^{(i_{l'})}(t)(\xi_{l'})) + \mathcal O_{d,\varphi, \rho_0,\beta,\theta,R,a,\varepsilon} (\varepsilon(L)^R L^{-d/2})
\]
where $\mathcal P_R$ is the set of partitions of $[|1,R|]$ whose elements are pairs of $[|1,R|]$. 
\end{lemma}

\begin{proof} We write that
\[
 \E \Big( \prod_{l=1}^R \hat {\mathcal U}_M^{(i_l)}(\xi_l)(t)\Big) =\sum_{n_l\leq M}  \E \Big( \prod_{l=1}^R \hat {u}_{n_l}^{(i_l)}(\xi_l)(t)\Big).
\]
We use the result \eqref{Eq:explicitComp} to get that for $L$ big enough
\[
\Big|\E \Big( \prod_{l=1}^R \hat {u}_{n_l}^{(i_l)}(\xi_l)(t)\Big) - \sum_{\mathcal O \in \mathcal P_R}\prod_{\{l,l'\} \in \mathcal O}\E( \hat {u}_{n_l}^{(i_l)}(t)(\xi_l)\hat{u}^{(i_{l'})}_{n_{l'}}(t)(\xi_{l'}))\Big| \lesssim_R L^{-d/2}2^{-\sum n_l} \varepsilon(L)^R.
\]
We sum back on $n_l$ to conclude.
\end{proof}

We now use Lemmas \ref{lem:removingRM}, \ref{lem:genEst}, and \ref{lem:removeEL} with $R = 2$ to get that
\[
\E( \hat {\mathcal U}_{M}^{(i_l)}(t)(\xi_l)\hat{\mathcal U}^{(i_{l'})}_{M}(t)(\xi_{l'})) = \E( {\bf 1}_{\mathcal E_L}\hat u(t)^{(i_l)}(\xi_l)\hat u^{(i_{l'})}(t)(\xi_{l'}))+ \mathcal O_{d,\varphi, \rho_0,\beta,\theta,R,a,\varepsilon} (\varepsilon(L)^2 L^{-d/2})
\]
and conclude.

\appendix

\section{Estimates on the norm of the initial datum}\label{app:sizeofaL}

For this appendix, we set
\[
a_L(x) = \sum_{k\in \Z^d_*} \frac{e^{ikx/L}}{(2\pi L)^{d/2}} a(k/L) g_k
\]
where $a$ is even. For $\xi \in ([-1,1]\smallsetminus \{0\})^d$, we assume that $a(\xi)$ lies in the orthogonal of $\{\xi\}$ and is of norm $1$. Outside, we assume that $a=0$.

We have seen in Proposition \ref{prop:estDI} that for $R$ big enough
\[
\mathbb P(\|a_L\|_{\rho_0}\geq R\sqrt{\ln L}) \leq e^{-cR^2\ln L}.
\]

We now want to check that for $\delta$ small enough if $a\neq 0$, 
\[
\mathbb P(\|a_L\|_{\rho_0}> \delta \sqrt{\ln L}) \geq \frac12
\]
which would tell us that $\sqrt{\ln L}$ is the typical size of $\|a_L\|_{\rho_0}$. We note that this typical size is more due to the number of independent Gaussian variables we sum that to the size of the box.

We have that 
\[
\|a_L\|_{L^\infty} = \|\sum_n \phi_n * a_L \|_{L^\infty} \leq \sum_n \|\phi_n*a_L\|_{L^\infty} \leq \|a_L\|_{\rho_0}
\]
and thus
\[
\mathbb P(\|a_L\|_{\rho_0}\geq \delta \sqrt{\ln L} )\geq \mathbb P(\|a_L\|_{L^\infty} \geq \delta \sqrt{\ln L}).
\]

We fix for $n\in [|1,L|]^d$, $x_n = 2n\pi$. We have 
\[
\mathbb P(\|a_L\|_{\rho_0}\geq \delta \sqrt{\ln L} )\geq \mathbb P(\exists n,\;|a_L(x_n)| \geq \delta \sqrt{\ln L}).
\]
We have that
\[
\mathbb P(\exists n,\;|a_L(x_n)| \geq \delta \sqrt{\ln L}) = 1 - \mathbb P(\forall n ,\; |a_L(x_n)|< \delta \sqrt{\ln L}).
\]
We have 
\[
\E(\an{ a_L(x_n), a_L(x_m)}_{\C^d}) =\frac1{(2\pi L)^d} \sum_{k\in ([|-L,L|]\smallsetminus\{0\})^d} e^{ik(x_n-x_m)/L} = \frac1{\pi^d}\delta_{n,m}.
\]
Therefore,
\[
\mathbb P(\forall n ,\; |a_L(x_n)|\leq \delta \sqrt{\ln L}) = \mathbb P(|a_L(x_0)|< \delta \sqrt{\ln L})^{L^d}.
\]
We have that
\[
 \mathbb P(|a_L(x_0)|< \delta \sqrt{\ln L}) = 1 - \mathbb P(|a_L(x_0)|\geq \delta \sqrt{\ln L}).
\]
The random variable $a_L(x_0)$ is a real Gaussian variable with variance $\pi^{-d}$ we deduce that for $L$ big enough,
\[
\mathbb P(|a_L(x_0)|\geq \delta \sqrt{\ln L}) \geq e^{-2\pi^{-d}\delta^2 \ln L} = L^{-\tilde \delta}
\]
with $\tilde \delta = 2\pi^{-d}\delta^d$.

We deduce that 
\[
 \mathbb P(|a_L(x_0)|< \delta \sqrt{\ln L}) \leq (1 - L^{-\tilde \delta})^{L^d}.
\]
If $\tilde \delta < d$, we have that
\[
(1 - L^{-\tilde \delta})^{L^d} \rightarrow 0
\]
as $L\rightarrow \infty$, and thus for $L$ big enough
\[
\mathbb P(\|a_L\|_{\rho_0} \geq \delta \sqrt{\ln L}) \geq \frac12.
\]

\section{Polish notations}\label{app:Polish}

\begin{definition} Let $n\in \N$. Let $\bar{\mathcal A}_n$ be the subset of $\{0,1\}^{nN+1}$ such that if $(a_1,\hdots,a_{nN+1}) \in \bar{\mathcal A}_n$ we have that for all $k<nN+1$,
\[
k < (\sum_{j\leq k} a_j) N+1
\]
and such that $\sum_{j=1}^{nN+1} a_j = n$.
\end{definition}

\begin{remark} The second condition means that there are exactly $n$ ones in the sequence. It can be rewritten as
\[
nN+1 = (\sum_{j\leq nN+1} a_j) N+1.
\]
\end{remark}

\begin{remark} For $n=0$, we have 
\[
\bar{\mathcal A}_0 = \{0\}.
\]
For $n=1$, we have that
\[
1 < a_1 N +1,
\]
hence $a_1 = 1$. We get 
\[
\bar{\mathcal A}_1 = \{(1,0,\hdots,0)\}.
\]
\end{remark}

\begin{notation} If $M\in \N$, $n_j\in \N)$ and $(A_1,\hdots,A_M) \in \prod_{j\leq M} \{0,1\}^{n_j}$, we write
\[
A_1\hdots A_M = (a_{1,1},\hdots,a_{1,n_1},a_{1,1},\hdots,a_{2,n_2},\hdots, a_{M,1},\hdots,a_{M,n_M})
\]
the contatenation of $A_1$ up to $A_M$ with
\[
A_j = (a_{j,1},\hdots,a_{j,n_j})
\]
for all $j$.
\end{notation}

\begin{proposition} Let $n\in \N$. Let $A \in \bar{\mathcal A}_{n+1}$. There exists $(n_1,\hdots,n_N) \in \N^{N}$ and $(A_1,\hdots,A_N) \in \prod_j \bar{\mathcal A_j}$ such that
\[
n= \sum_{j=1}^N n_j,\quad A = 1A_1\hdots A_N;
\]
this decomposition is unique.

Conversely, if $A \in \{0,1\}^{nN+1}$ is equal to 
\[
1A_1\hdots A_N
\]
with $A_j \in \bar{\mathcal A}_{n_j}$ for all $j$ and such that $\sum n_j = n$ then
\[
A \in \bar{\mathcal A}_{n+1}.
\]
\end{proposition}

\begin{proof} Let $n\in \N$ and $A \in \bar{\mathcal A}_{n+1}$. We write
\[
A = (a_1,\hdots,a_{(n+1)N+1}), \quad \forall k \leq nN+1, \; b_k = \sum_{j\leq k} a_j.
\]

Since 
\[
1 < a_1 N +1,
\]
we deduce that $a_1 = 1$. We set for $m=1,\hdots, N$,
\[
k_m = \min (k | k\geq (b_k-1)N+m+1).
\]
We prove that $(k_m)_m$ is well-defined, strictly increasing and that for all $m$, $k_m = (b_{k_m}-1)N +m+1$.

First of all, we have 
\[
(n+1)N+1 = b_{nN+1}N+1 = (b_{nN+1}-1)N+ N +1 \geq (b_{nN+1}-1)N +m +1.
\]
Hence $k_m$ is well-defined.

Set $c_{k,m} = (b_k-1)N+m+1 - k$. We have that
\[
c_{k+1,m} - c_{k,m} = (b_{k+1}-b_k)N -1
\]
and because $(b_k)$ is increasing, this ensures that
\[
c_{k+1,m} - c_{k,m} \leq -1.
\]
We deduce that we cannot pass from a (strictly) positive $c_{k,m}$ to a (strictly) negative $c_{k+1,m}$. 

We have $c_{1,1} = 1 > 0$ thus $k_1 >1$, $c_{1,k_1} = 0$ and thus $c_{2,k_1}> 0$. By induction, we get that $(k_m)$ is strictly increasing and that for all $m$, $c_{k_m,m} = 0$.

What is more,
\[
k_N = \min (k| k\geq b_k N +1) = (n+1)N+1
\]
by definition of $\bar{\mathcal A}_{n+1}$.

We set $\tilde n_m = k_m -k_{m-1}$ with the convention $k_0 = 1$. We write
\[
A = (1,a_{1,1},\hdots, a_{1,\tilde n_1},\hdots, a_{N,1},\hdots,a_{N,\tilde n_N})
\]
and 
\[
A_m = (a_{m,1},\hdots,a_{m,\tilde n_j}).
\]
We have $A =1A_1\hdots A_N$. 

We also have that $\tilde n_m = (b_{k_m}-b_{k_{m-1}})N + 1$ hence $\tilde n_m = n_m N +1$ with $n_m = b_{k_m} - b_{k_{m-1}}$. We prove that
\[
A_m \in \bar{\mathcal A}_{n_m}.
\]

We have $a_{m,k} = a_{k+k_{m-1}}$ and $b_{m,k} = \sum_{j\leq k} a_{m,j} = b_k - b_{k_{m-1}}$.

We use the definition of the sequence $(k_m)_m$ to that for $k< n_m N+1$,
\[
k + k_{m-1} < (b_{k+k_{m-1}} - 1) N + m+1.
\]
Using that $k_{m-1} = (b_{k_{m-1}}-1) N +m$, we get
\[
k < (b_{k+k_{m-1}} - b_{k_{m-1}})N +1 = b_{m,k}N +1
\]
and since $b_{m,n_mN+1} = n_m$ by definition, we have $A_m \in \bar{\mathcal A}_m$. 

The construction of the decomposition ensures its uniqueness.

Conversely, if $A = 1A_1\hdots A_N$ with $A_m \in \bar{\mathcal A}_{n_m}$ and $\sum n_m = n$ then $A \in \bar{\mathcal A}_{n+1}$. Indeed, for $k = 1$, we have
\[
1=k < b_1 N+1 = N+1.
\]
and for $k\in [k_{m-1}+1,k_{m}]$ with $k_0 = 1$, $k_m = \sum_{l\leq m}n_l N +m +1$, we have 
\[
b_k = \sum_{l< m} n_l + 1 + b_{m,k-k_{m-1}}.
\]
We deduce 
\[
b_k N +1 = \sum_{l<m} n_l N  + N b_{m, k-k_{m-1}}  +N+1 = k_{m-1} - (m-1) + N b_{m, k-k_{m-1}} +N.
\]
If $m<N$, then we use that $N b_{m, k-k_{m-1}}\geq k-k_{m-1} -1$, and get
\[
b_k N +1 \geq N-m + k > k.
\]
If $m=N$ then we have 
\[
b_k N +1 = k_{N-1}  + N b_{N, k-k_{N-1}} +1.
\]
If $k<(n+1)N+1$ then $k-k_{N-1} < (n+1)N+1 - N -\sum_{l< N} n_l N = n_N N+1$, and thus
\[
b_k N+1 > k_{N-1} + k - k_{N-1} = k.
\]
If $k = (n+1)N+1$ then $k-k_{N-1} = n_N N +1$ and thus
\[
b_k N+1 = k.
\]
which concludes the proof.
\end{proof}

We deduce that $\bar{\mathcal A_n}$ is isomorphic to $\mathcal A_n$ and in particular they have the same cardinal.

\bibliographystyle{plain}
\bibliography{WTbib}

\end{document}